\documentclass[11pt]{article}
\usepackage[margin=1in]{geometry}
\pdfoutput=1

\usepackage{amsmath}
\usepackage{amsfonts}
\usepackage{amsthm}
\usepackage{graphicx}
\usepackage{color}
\usepackage[usenames,dvipsnames]{xcolor}
\usepackage{hyperref}
\hypersetup{colorlinks=true,
    linkcolor=BrickRed,
    citecolor=OliveGreen,
    urlcolor=MidnightBlue}
\usepackage{booktabs}
\newcommand{\ra}[1]{\renewcommand{\arraystretch}{#1}}
\newcommand{\Itwo}{\begin{bmatrix} 1 & \;0\,\\0 & \;1\,\end{bmatrix}}
\newcommand{\Rtwo}{\begin{bmatrix} 1 & 0\\0 & \!\!-1\end{bmatrix}}

\newcommand{\R}{\mathbb{R}}

\newcommand{\Z}{\mathbb{Z}}
\newcommand{\Sym}{\mathcal{S}}

\newcommand{\conv}{\textup{conv}}
\newcommand{\symdiff}{\Delta}
\newcommand{\PP}{\textup{PP}}
\newcommand{\Cube}{\textup{C}}

\newcommand{\Spin}{\textup{Spin}}

\newcommand{\diag}{\textup{diag}}
\newcommand{\Diag}{\mathcal{D}}
\newcommand{\Cl}{\textup{Cl}}
\newcommand{\Ieven}{\mathcal{I}_{\text{even}}}
\newcommand{\conj}[1]{\overline{#1}}
\newcommand{\cliffid}{\mathbf{1}}

\newcommand{\tildeQ}{\widetilde{Q}}
\newcommand{\projto}[1]{\pi_{#1}}
\newcommand{\projfrom}[1]{\pi_{#1}^*}
\newcommand{\projtofrom}[1]{\projfrom{#1}\projto{#1}}

\newcommand{\parity}{\textup{parity}}
\newcommand{\psdlift}{PSD lift}
\newcommand{\psdlifts}{PSD lifts}

\newcommand{\psd}{\succeq}

\newcommand{\nsd}{\preceq}

\DeclareMathOperator*{\tr}{\textup{tr}}

\theoremstyle{plain}
\newtheorem{thm}{Theorem}[section]
\newtheorem{proposition}[thm]{Proposition}
\newtheorem{lemma}[thm]{Lemma}

\newtheorem{theorem}[thm]{Theorem}
\newtheorem{corollary}[thm]{Corollary}
\theoremstyle{definition}
\newtheorem{example}[thm]{Example}
\newtheorem{definition}[thm]{Definition}
\theoremstyle{remark}
\newtheorem{remark}[thm]{Remark}

\usepackage{subfig}
\usepackage{caption}

\title{Semidefinite descriptions of the convex hull of rotation matrices}
\author{James Saunderson \and Pablo A.~Parrilo \and Alan S.~Willsky
    \thanks{The authors are with the Laboratory for Information and Decision Systems, Department of Electrical 
        Engineering and Computer Science, Massachusetts Institute of Technology, Cambridge MA 02139, USA. 
        Email: \{jamess,parrilo,willsky\}@mit.edu.}}

\begin{document}
\maketitle
\begin{abstract}
    We study the convex hull of $SO(n)$, thought of as the set of $n\times n$
    orthogonal matrices with unit determinant, from the point of view of
    semidefinite programming. We show that the convex hull of $SO(n)$ is
    \emph{doubly spectrahedral}, i.e.~both it and its polar have a description
    as the intersection of a cone of positive semidefinite matrices with an
    affine subspace.  Our spectrahedral representations are explicit, and are
    of minimum size, in the sense that there are no smaller spectrahedral
    representations of these convex bodies. 
\end{abstract}
\section{Introduction}
\label{sec:intro}
Optimization problems where the decision variables are constrained to be in the
set of orthogonal matrices
\begin{equation}
    \label{eq:ondef}
    O(n) := \{X\in \R^{n\times n}:\; X^TX = I\}
\end{equation}
arise in many contexts (see, e.g., \cite{nemirovski2007sums,naor2012efficient}
and references therein), particularly when searching over Euclidean isometries
or orthonormal frames.  In some situations, especially those arising from
physical problems, we require the additional constraint that the decision
variables be in the set of rotation matrices
\begin{equation}
    \label{eq:sondef}
    SO(n) := \{X\in \R^{n\times n}:\; X^TX = I,\;\;\det(X) = 1\}
\end{equation}
representing Euclidean isometries that also \emph{preserve orientation}. For
example, these additional constraints arise in problems involving attitude
estimation for spacecraft \cite{psiaki2009generalized} or pose estimation in
computer vision applications \cite{horowitz2014}, or in understanding protein
folding \cite{longinetti2010convex}. The unit determinant constraint is
important in these situations because we typically cannot reflect physical
objects such as spacecraft or molecules.

The set of $n\times n$ rotation matrices is non-convex, so optimization
problems over rotation matrices are ostensibly non-convex optimization
problems. An important approach to global non-convex optimization is to
approximate the original non-convex problem with a tractable convex
optimization problem. In some circumstances it may even be possible to
\emph{exactly reformulate} the original non-convex problem as a tractable
convex problem.  This approach to global optimization via convexification has
been very influential in combinatorial optimization
\cite{schrijver2003combinatorial}, and more generally in polynomial
optimization via the machinery of moments and sums of squares
\cite{blekherman2013semidefinite}.
As an example of a problem amenable to this approach, in Section~\ref{sec:app}
we describe the problem of jointly estimating the attitude and spin-rate of a
spinning satellite and show how to reformulate this ostensibly non-convex
problem as a convex optimization problem that, using the constructions in this
paper, can be expressed as a semidefinite program. 

When we attempt to convexify optimization problems involving rotation matrices
two natural geometric objects arise. The first of these is the \emph{convex
hull} of $SO(n)$ which we denote, throughout,  by $\conv\,SO(n)$.  The second
convex body of interest in this paper is the \emph{polar} of $SO(n)$, the set
of linear functionals that take value at most one on $SO(n)$, i.e.,
\[ SO(n)^\circ = \{Y\in \R^{n\times n}: \langle Y,X\rangle \leq 1\;\text{for all $X\in SO(n)$}\}\]
where we have identified $\R^{n\times n}$ with its dual space via the trace
inner product $\langle Y,X\rangle = \tr(Y^TX)$.  These two convex bodies are
closely related. Since $\conv\,SO(n)$ is closed and contains the origin it
follows from basic results of convex analysis \cite[Theorem
14.5]{rockafellar1997convex} that $\conv\,SO(n) = {SO(n)^{\circ}}^\circ$.

We also study the convex hull and the polar of orthogonal matrices in this
paper. It is well-known that these correspond to commonly used matrix norms
(see, e.g.,~\cite{sanyal2011orbitopes}).  The convex hull of $O(n)$ is the
\emph{operator norm ball}, the set of $n\times n$ matrices with largest
singular value at most one, and the polar of $O(n)$ is the \emph{nuclear norm
ball}, the set of $n \times n$ matrices such that the sum of the singular
values is at most one, i.e.\
\[ \conv\,O(n) = \left\{X\in \R^{n\times n}: \sigma_1(X) \leq 1\right\}\quad\text{and}\quad
O(n)^\circ = \left\{X\in \R^{n\times n}: \sum_{i=1}^{n}\sigma_i(X) \leq 1\right\}.\]

Note that $O(n)$ is the (disjoint) union of $SO(n)$ and the set $SO^-(n) :=
\{X\in \R^{n\times n}: X^TX = I, \det(X)=-1\}$.  As such, it follows from basic
properties of the polar \cite[Corollary 16.5.2]{rockafellar1997convex} that  
\begin{equation}
    \label{eq:onpolar-intersection}
    O(n)^\circ = SO(n)^\circ \cap SO^-(n)^\circ
\end{equation}
allowing us to deduce properties of $O(n)^\circ$ from those of $SO(n)^\circ$.
On the other hand we show in Proposition~\ref{prop:son2} that
\begin{equation}
    \label{eq:convson-intersection}
    \conv\,SO(n) = \conv\,O(n) \cap (n-2)SO^-(n)^\circ
\end{equation}
allowing us to deduce properties of $\conv\,SO(n)$ from properties of
$\conv\,O(n)$ and $SO^-(n)^\circ$. Figure~\ref{fig:sononpic} illustrates the
differences between $\conv\,SO(n)$ and $\conv\,O(n)$ and the relationship
described in~\eqref{eq:onpolar-intersection}.

\begin{figure}
\begin{center}
    \subfloat[][A $2$-dimensional projection of 
    $\conv\,SO(3)$ (red), $\conv\,SO^-(3)$ (blue), and $\conv\,O(3) = \conv\,{[}SO(3)\cup SO^-(3){]}$ (black).]
    {\includegraphics[scale=0.35]{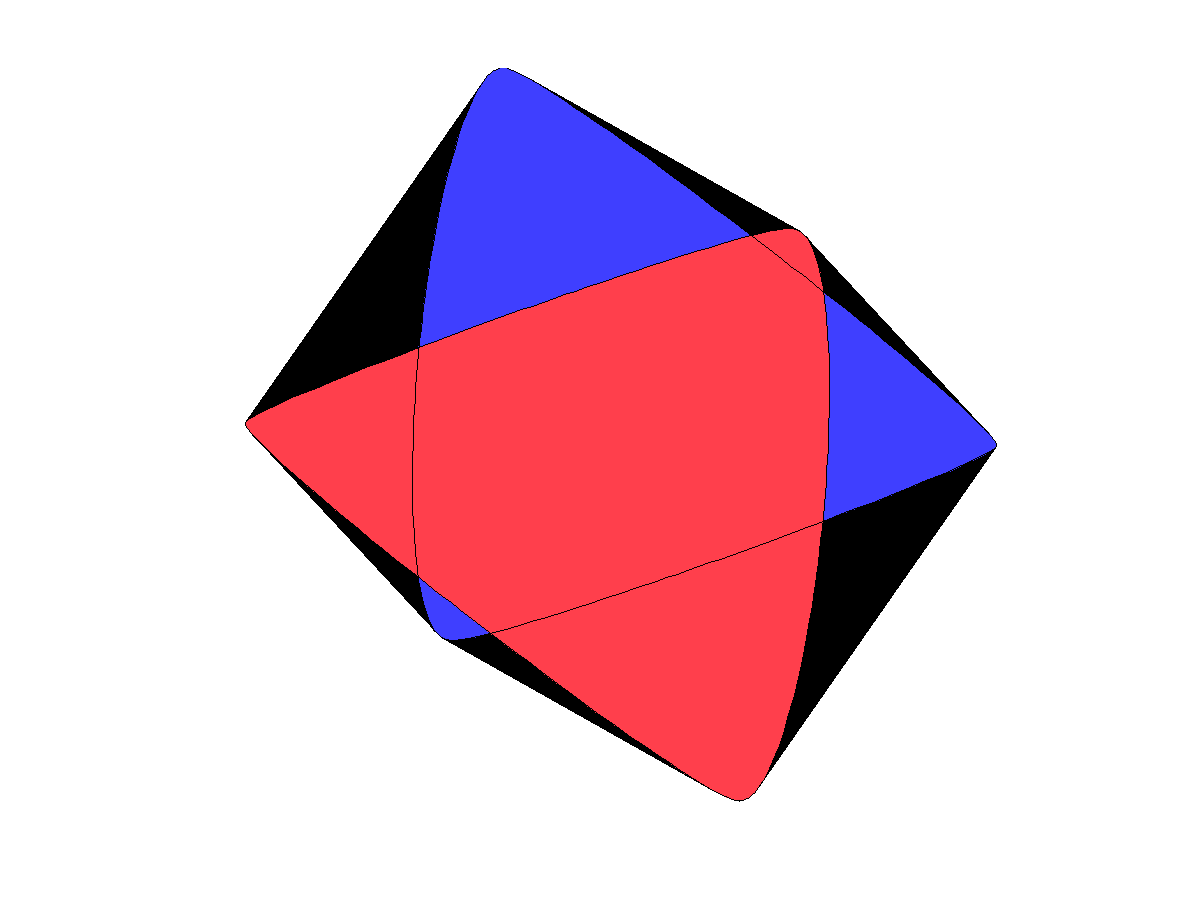}\label{fig:conv}}
    \hfill\subfloat[][The corresponding $2$-dimensional section of $SO(3)^\circ$ (red), $SO^-(3)^\circ$ (blue),
    and $O(3)^\circ = SO(3)^\circ \cap SO^-(3)^\circ$ (black). ]{\includegraphics[scale=0.35]{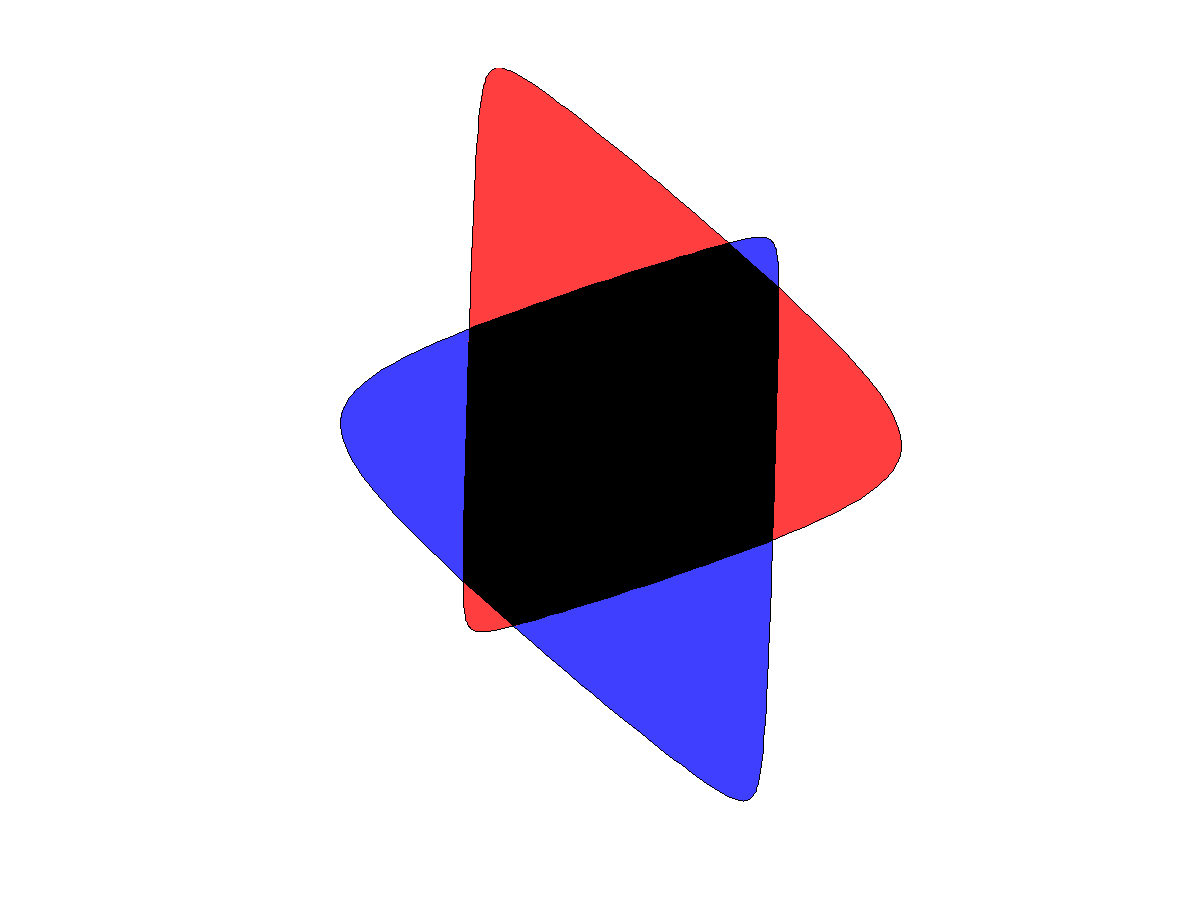}\label{fig:polar}}
\end{center}
\caption{\label{fig:sononpic} Pictures of some of the convex bodies considered in 
    this paper. These were created by optimizing $100$ linear functionals over
    each of these sets to obtain $100$ boundary points. The optimization was
    performed by implementing our spectrahedral representations in the parser
    YALMIP \cite{yalmip}, and solving the semidefinite programs numerically
using SDPT3 \cite{toh1999sdpt3}.}
\end{figure}
The convex bodies $\conv\,SO(n)$ and $\conv\,O(n)$ are examples of
\emph{orbitopes}, a family of highly symmetric convex bodies that arise from
representations of groups
\cite{sanyal2011orbitopes,barvinok1988convex,barvinok2005convex}. Suppose a
compact group $G$ acts on $\R^n$ by linear transformations and $x_0\in \R^n$.
Then the \emph{orbit} of $x_0$ under $G$ is
\[ G\cdot x_0 = \{g\cdot x_0: g\in G\} \subset\R^n\]
and the corresponding \emph{orbitope} is $\conv\,(G\cdot x_0)$, the convex hull
of the orbit.  The sets $O(n)$ and $SO(n)$ defined above can be thought of as
the orbit of the identity matrix $I\in \R^{n\times n}$ under the linear action
of the groups $O(n)$ and $SO(n)$, respectively, by right multiplication on
$n\times n$ matrices.  The corresponding orbitopes are known as the
\emph{tautological $O(n)$ orbitope} and the \emph{tautological $SO(n)$
orbitope} respectively \cite{sanyal2011orbitopes}.  The set $SO^-(n)$ can be
viewed as the orbit of $R:=\diag^*(1,1,\ldots,1,-1)$, the diagonal matrix with
diagonal entries $(1,1,\ldots,1,-1)$, under the same $SO(n)$ action on $n\times
n$ matrices.  Note that $SO^-(n)$ is then the image of $SO(n)$ under the
invertible linear map $X\mapsto R\cdot X$.

\paragraph{Spectrahedra}
For convex reformulations or relaxations involving the convex hull of $SO(n)$
to be useful from a computational point of view we need an effective
description of the convex body $\conv\,SO(n)$.  One effective way to describe a
convex body is to express it as the intersection of the cone of symmetric
positive semidefinite matrices with an affine subspace. Such convex bodies are
called \emph{spectrahedra} \cite{ramana1995some} and are natural
generalizations of polyhedra. Algebraically, a convex subset $C$ of $\R^n$
(containing the origin in its interior\footnote{We can assume this without loss
of generality by translating $C$ and restricting to its affine hull}) is a
spectrahedron if it can be expressed as the feasible region of a linear matrix
inequality of the form
\begin{equation}
    \label{eq:defspec}
C = \left\{x\in \R^n: I_m + \sum_{i=1}^{n} A_i x_i \psd 0\right\}
\end{equation}
where $I_m$ is the $m\times m$ identity matrix, $A_1,A_2,\ldots,A_n$ are
$m\times m$ real symmetric matrices and $M \psd 0$ means that $M$ is a
symmetric positive semidefinite matrix.  If the matrices $A_i$ are $m\times m$,
we call the description \eqref{eq:defspec} a \emph{spectrahedral representation
of size $m$}.

Giving a spectrahedral representation for a convex set has algebraic,
geometric, and algorithmic implications. Algebraically, a spectrahedral
representation of $C$ of size $m$ as in \eqref{eq:defspec} tells us that the
degree $m$ polynomial  $p(x) = \det(I + \sum_{i=1}^{n}A_i x_i)$ vanishes on the
boundary of $C$, and that $C$ itself can be written as the region defined by
$m$ polynomial inequalities (i.e.~it is a basic closed semi-algebraic set)
\cite[Theorem 20]{renegar2006hyperbolic}.  Geometrically, a spectrahedral
representation of $C$ gives information about its facial structure. For
example, it is known that all faces of a spectrahedron are exposed (i.e., can
be obtained as the intersection of the spectrahedron with a supporting
hyperplane), since the same is true for the positive semidefinite cone.

From the point of view of optimization, problems involving minimizing a linear
functional over a spectrahedron are called \emph{semidefinite optimization}
problems \cite{blekherman2013semidefinite} and are natural generalizations of
the more well-known class of linear programming problems. Semidefinite
optimization problems can be solved (to any desired accuracy) in time
polynomial in $n$ and $m$.

The convex sets that can be obtained as the images of spectrahedra under linear
maps are also of interest. Indeed to minimize a linear functional over a
projection of a spectrahedron, one can simply lift the linear functional and
minimize it over the spectrahedron itself using methods for semidefinite
optimization.  We say a convex body has a \emph{\psdlift{}} if it has a
description as a projection of a spectrahedron (see Section~\ref{sec:eqpsd}).
\psdlifts{} are important because they form a strictly larger family of convex
sets than spectrahedra, and because some spectrahedra have \psdlifts{} that are
much more concise than their smallest spectrahedral representations
(generalizing the notion of extended formulations for polyhedra). On the other
hand convex bodies that have \psdlifts{} do not enjoy the same nice algebraic
and geometric properties as spectrahedra---indeed they are semialgebraic but
not necessarily basic semialgebraic, and are not necessarily facially
exposed~\cite{blekherman2013semidefinite}.

Throughout much of the paper we consider only spectrahedral representations,
confining our discussion of \psdlifts{} to Section~\ref{sec:eqpsd}.

\paragraph{Doubly spectrahedral convex sets}
In this paper we are interested in both $SO(n)^\circ$ and $\conv\,SO(n)$, and
so study both from the point of view of semidefinite programming. For finite
sets $S$, both $S^\circ$ and $\conv\,S$ are polyhedra. On the other hand, for
infinite sets $S$, usually neither $S^\circ$ nor $\conv\,S$ are spectrahedra.
Even if a convex set is a spectrahedron, typically its polar is not a
spectrahedron (see Section~\ref{sec:conclusion}).  We use the term \emph{doubly
spectrahedral convex sets} to refer to those very special convex sets $C$ with
the property that both $C$ and $C^\circ$ are spectrahedra.

\paragraph{Main contribution} 
The main contribution of this paper is to establish that $\conv\,SO(n)$ is
doubly spectrahedral and to give explicit spectrahedral representations of both
$SO(n)^\circ$ and $\conv\,SO(n)$. 

\paragraph{Main proof technique} The main idea behind our representations is
that we start with a \emph{parameterization} of $SO(n)$, rather than working
with the defining equations in~\eqref{eq:sondef}.  The parameterization is a
direct (and classical) generalization of the widely used unit quaternion
parameterization of $SO(3)$.  In higher dimensions the unit quaternions are
replaced with $\Spin(n)$, a multiplicative subgroup of the invertible elements
of a Clifford algebra. In the cases $n=2$ and $n=3$ it is relatively
straightforward to produce our semidefinite representations directly from this
parameterization.  For $n\geq 4$ the parameterization does not immediately
yield our semidefinite representations. The additional arguments required to
establish the correctness of our representations for $n\geq 4$ form the main
technical contribution of the paper.

\subsection{Statement of results}
\label{sec:results}

In this section we explicitly state the spectrahedral representations that we
prove are correct in subsequent sections of the paper. In particular we state
spectrahedral representations for $SO(n)^\circ$ and $\conv\,SO(n)$, as well as
a spectrahedral representation of $O(n)^\circ$, the nuclear norm ball.  All the
spectrahedral representations stated in this section are of minimum size (see
Theorem~\ref{thm:lower}). The reader primarily interested in implementing our
semidefinite representations should find all the information necessary to do so
in this section.

\paragraph{Matrices of the spectrahedral representations}
Our main results are stated in terms of a collection of symmetric
$2^{n-1}\times 2^{n-1}$ matrices denoted $(A_{ij})_{1\leq i,j\leq n}$. We give
concrete descriptions of them here in terms of the Kronecker product of
$2\times 2$ matrices, deferring more invariant descriptions to
Appendix~\ref{app:clif}.  The matrices $A_{ij}$ can be expressed as
\begin{equation}
    \label{eq:AB}
    A_{ij} = -P_{\textup{even}}^T\lambda_i \rho_jP_{\textup{even}}
\end{equation}
where $(\lambda_i)_{i=1}^{n}$ and $(\rho_i)_{i=1}^n$ are the $2^n\times 2^n$
skew-symmetric matrices defined concretely by 
\begin{align*}
    \lambda_i & = \overbrace{\Rtwo \otimes \cdots \otimes \Rtwo}^{i-1} \otimes \begin{bmatrix} 0 & -1\\1 & 0\end{bmatrix}\otimes 
    \overbrace{\Itwo \otimes \cdots \otimes \Itwo}^{n-i}\\
    \rho_i & = \underbrace{\Itwo \otimes \cdots \otimes \Itwo}_{i-1} \otimes \begin{bmatrix} 0 & -1\\1 & 0\end{bmatrix}\otimes 
    \underbrace{\Rtwo \otimes \cdots \otimes \Rtwo}_{n-i}\\
\end{align*}
and $P_{\textup{even}}$ is the $2^{n}\times 2^{n-1}$ matrix with orthonormal columns 
\begin{align*}
    P_{\textup{even}} & = \frac{1}{2}\begin{bmatrix} I_{2^{n-1}} + \Rtwo\otimes \cdots \otimes \Rtwo\\
    \phantom{\cdots}\\
        I_{2^{n-1}} - \Rtwo \otimes \cdots \otimes\Rtwo\end{bmatrix}.
\end{align*}
Note that $P_{\textup{even}}^TMP_{\textup{even}}$ just selects a particular
$2^{n-1}\times 2^{n-1}$ principal submatrix of $M$.  Since, for any pair $1\leq
i,j\leq n$, $\lambda_i$ and $\rho_j$ are skew symmetric and commute it follows
that each $A_{ij}$  is symmetric. Furthermore since $\lambda_i$ and $\rho_j$
are signed permutation matrices, so is $-\lambda_i\rho_j$. From this we can see
that all of the entries of the $A_{ij}$ are $0$, $1$, or $-1$. 

    \paragraph{Spectrahedral representations}
    The following, which we prove in Section~\ref{sec:reps}, is the main
    technical result of the paper.
    \begin{theorem}
        \label{thm:sonpolar}
        The polar of $SO(n)$ is a spectrahedron.  Explicitly
        \begin{equation}
            \label{eq:sonpolar}
            SO(n)^\circ = 
        \bigg\{Y\in \R^{n\times n}: \sum_{i,j=1}^{n} A_{ij}Y_{ij} \nsd I_{2^{n-1}}\bigg\}
    \end{equation}
        where the $2^{n-1}\times 2^{n-1}$ matrices $A_{ij}$ are defined in
        \eqref{eq:AB}.  
    \end{theorem}
    Since $O(n) = SO(n)\cup SO^-(n)$ as a corollary of Theorem~\ref{thm:sonpolar} we obtain a spectrahedral representation
    of $O(n)^\circ = SO(n)^\circ \cap SO^-(n)^\circ$. 
    \begin{theorem}
        \label{thm:onpolar}
        The polar of $O(n)$ is a spectrahedron. Explicitly
        \[ O(n)^\circ = \bigg\{Y\in \R^{n\times n}: \sum_{i,j=1}^{n}
    A_{ij}Y_{ij} \nsd I_{2^{n-1}},\;\sum_{i,j=1}^{n} A_{ij}[RY]_{ij} \nsd I_{2^{n-1}}\bigg\}.\]
        where $R = \diag^*(1,1,\ldots,1,-1)$.
    \end{theorem}
    
    Just because a convex set $C$ is a spectrahedron does not, in general, mean
    that its polar is also spectrahedron (see Section~\ref{sec:conclusion} for
    a simple example). Even if we are in the special case where $C$ is doubly
    spectrahedral, we cannot simply dualize a spectrahedral representation of
    $C$ to obtain a spectrahedral representation of its polar.  Nevertheless,
    by a separate argument we can show that $\conv\,SO(n) = \conv\,O(n) \cap
    (n-2) SO^-(n)^\circ$ (Proposition~\ref{prop:son2}) to obtain a
    \emph{spectrahedral} representation of $\conv\,SO(n)$.
We explain how this works in detail in Section~\ref{sec:son}.
    \begin{theorem}
        \label{thm:son}
        The convex hull of $SO(n)$ is a spectrahedron. Explicitly
        \begin{align}
            \textup{conv}\;SO(n) & = \bigg\{X\in \R^{n\times n}: 
                \begin{bmatrix} 0& X\\X^T & 0\end{bmatrix} \nsd I_{2n},\;\;
        \sum_{i,j=1}^{n} A_{ij}[RX]_{ij}
        \nsd (n-2)I_{2^{n-1}}\bigg\}.\label{eq:son}
        \end{align}
        In the special cases $n=2$ and $n=3$ these representations can be
        simplified to 
        \begin{align}
        \conv\,SO(2) & = \bigg\{\begin{bmatrix} c & -s\\s & c\end{bmatrix}\in
    \R^{2\times 2}: \begin{bmatrix} 1+c & s\\s & 1-c\end{bmatrix} \psd
        0\bigg\}\quad\text{and}\label{eq:so2}\\
            \conv\,SO(3) & = \bigg\{X\in \R^{3\times 3}: \sum_{i,j=1}^{3}
A_{ij}[RX]_{ij}\nsd I_4\bigg\}\label{eq:so3}\\
& = \left\{X\in \R^{3\times 3}: \left[\begin{smallmatrix}
1-X_{11}-X_{22}+X_{33} &X_{13}+X_{31} &X_{12}-X_{21} &X_{23}+X_{32}\\
X_{13}+X_{31}&1+X_{11}-X_{22}-X_{33}&X_{23}-X_{32}&X_{12}+X_{21}\\
X_{12}-X_{21}&X_{23}-X_{32}&1+X_{11}+X_{22}+X_{33}&X_{31}-X_{13}\\
X_{23}+X_{32}&X_{12}+X_{21}&X_{31}-X_{13}&1-X_{11}+X_{22}-X_{33}\end{smallmatrix}\right] \psd 0\right\}.\nonumber
        \end{align}
    \end{theorem}
    We note that the representation of $\conv\,SO(3)$ described in Sanyal et
    al.~\cite[Proposition 4.1]{sanyal2011orbitopes} can be obtained from the
    spectrahedral representation for $\conv\,SO(3)$ given here by conjugating
    by a signed permutation matrix, establishing that the two representations
    are equivalent.  
    
    In Section~\ref{sec:lower} we prove that our spectrahedral representations
    in Theorems~\ref{thm:sonpolar},~\ref{thm:onpolar},~\ref{thm:son} are of
    minimum size. We do so by establishing lower bounds on the minimum size of
    spectrahedral representations of $SO(n)^\circ$, $\conv\,SO(n)$ and
    $O(n)^\circ$ that match the upper bounds given by our constructions.
    \begin{theorem}
        \label{thm:lower}
        If $n\geq 1$ the minimum size of a spectrahedral representation of
        $O(n)^\circ$ is $2^n$.
        If $n\geq 2$ the minimum size of a
        spectrahedral representation of $SO(n)^\circ$ is $2^{n-1}$.  If $n\geq
        4$ the minimum size of a spectrahedral representation of $\conv\,SO(n)$
        is $2^{n-1}+2n$. The minimum size of a spectrahedral representation of
        $\conv\,SO(3)$ is $4$. 
    \end{theorem}
    \paragraph{Representations as \psdlifts{}}
    Given a spectrahedral representation of size $m$ of a convex set $C$ (with
    the origin in its interior), by applying a straightforward conic duality
    argument (see, for example, \cite[Proposition 3.1]{gouveia2011positive}) we
    can obtain a \psdlift{} of $C^\circ$.  This representation, however, is
    usually \emph{not} a spectrahedral representation. 

    \begin{example}
        \label{eg:onpolarproj}
    Theorems~\ref{thm:onpolar} and~\ref{thm:lower} tell us that the smallest
    spectrahedral representation of $O(n)^\circ$, the nuclear norm ball, has
    size $2^n$. Yet by dualizing the size $2n$ spectrahedral representation of
    $\conv\,O(n)$ (given in Proposition~\ref{prop:on} to follow) we obtain a
    \psdlift{} of $O(n)^\circ$ of size $2n$
    \begin{align}
        \label{eq:projon} O(n)^\circ &= 
 \left\{Z\in \R^{n\times n}: \exists X,Y \;\;\text{s.t.}\;\; 
    \begin{bmatrix} X &Z\\Z^T & Y\end{bmatrix} \psd 0,\; \textup{tr}(X) + \textup{tr}(Y) = 2\right\}.\nonumber
\end{align}
    This is equivalent to the representation given by Fazel~\cite{fazel2002matrix} for the nuclear norm ball. 
\end{example}
    
    By dualizing, in a similar fashion, the spectrahedral representation of
    $SO(n)^\circ$ we obtain a representation of $\conv\,SO(n)$ as the
    projection of a spectrahedron, i.e.\ a \psdlift{} of $\conv\,SO(n)$. 
   \begin{corollary}
       \label{cor:sonproj}
       The convex hull of $SO(n)$ can be expressed as a projection of the
       $2^{n-1}\times 2^{n-1}$ positive semidefinite matrices with trace one as
       \[ \conv\,SO(n) = \left\{\begin{bmatrix} \langle A_{11},Z\rangle & \langle A_{12},Z\rangle & \cdots & \langle A_{1n},Z\rangle\\
                   \langle A_{21},Z\rangle & \langle A_{22},Z\rangle & \cdots & \langle A_{2n},Z\rangle\\
                   \vdots & \vdots & \ddots & \vdots \\
                   \langle A_{n1},Z\rangle & \langle A_{n2},Z\rangle & \cdots & \langle A_{nn},Z\rangle\end{bmatrix}: Z \psd 0, \; \tr(Z) = 1\right\}.\]
   \end{corollary}
   In some situations it may be preferable to use this representation of
   $\conv\,SO(n)$ rather than the spectrahedral representation in
   Theorem~\ref{thm:son}.

\subsection{Related work}
\label{sec:related-work}
That the convex hull of $O(n)$ is a spectrahedron is a classical result.  It
was not until recently that Sanyal et al.~\cite{sanyal2011orbitopes}
established that $O(n)^\circ$ is a spectrahedron by explicitly giving a
(non-optimal) size $\binom{2n}{n}$ spectrahedral representation. In the same
paper, Sanyal et al.~study numerous $SO(n)$- and $O(n)$-orbitopes considering
both convex geometric aspects such as their facial structure and Caratheodory
number, and algebraic aspects such as their algebraic boundary and whether they
are spectrahedra.  They describe (previously known) spectrahedral
representations of $\conv\,SO(2)$ and $\conv\,SO(3)$. The representation for
$\conv\,SO(3)$ given in \cite[Eq.~4.1]{sanyal2011orbitopes} is equivalent to
our representation in Theorem~\ref{thm:son}, and the representation given in
\cite[Eq.~4.2]{sanyal2011orbitopes} is equivalent to 
\[\conv\,SO(3) = \left\{\left[\begin{smallmatrix}Z_{11}-Z_{22}-Z_{33}+Z_{44} & -2Z_{13}-2Z_{24} & -2Z_{12}+2Z_{34}\\
            2Z_{13}-2Z_{24} & Z_{11} + Z_{22} - Z_{33} - Z_{44} & -2Z_{14}-2Z_{23}\\
            2Z_{12}+2Z_{34} & 2Z_{14}-2Z_{23} & Z_{11}-Z_{22}+Z_{33} - Z_{44} \end{smallmatrix}\right]: Z \psd 0, \;\tr(Z)=1\right\}\]
which can be obtained by specializing Corollary~\ref{cor:sonproj}.  Sanyal et
al.~raise the general question of whether $\conv\,SO(n)$ is a spectrahedron for
all $n$ (which we answer in the affirmative), and more broadly ask for a
classification of the $SO(n)$-orbitopes that are spectrahedra. 

Earlier work on orbitopes in the context of convex geometry includes the work
of Barvinok and Vershik \cite{barvinok1988convex} who consider orbitopes of
finite groups in the context of combinatorial optimization, Barvinok and
Blekherman \cite{barvinok2005convex}, who used asymptotic volume computations
to show that there are many more non-negative polynomials than sums of squares
(among other things), and  Longinetti et al.~\cite{longinetti2010convex} who
studied $SO(3)$-orbitopes with a view to applications in protein structure
determination. More recently Sinn~\cite{sinn2013algebraic} has studied in
detail the algebraic boundary of four-dimensional $SO(2)$-orbitopes as well as
the Barvinok-Novik orbitopes.

\subsection{Notation}
\label{sec:notation}
In this section we gather notation not explicitly defined elsewhere in the paper. 
We use $\Sym^m$ and $\Sym_+^m$ to denote the space of symmetric $m\times m$ matrices
and the cone of positive semidefinite matrices respectively. If $\mathcal{U}\subset \R^n$ is a 
subspace then $\projto{\mathcal{U}}:\R^n\rightarrow \mathcal{U}$ is the orthogonal projector
onto $\mathcal{U}$ and $\projfrom{\mathcal{U}}:\mathcal{U}\rightarrow \R^n$ is its adjoint.
If the subspace in question is the subspace of diagonal matrices $\mathcal{D}\subset \R^{n\times n}$ we occasionally also use $\diag := \projto{\mathcal{D}}$
and $\diag^* := \projfrom{\mathcal{D}}$. We frequently use the matrix $R = \diag^*(1,1,\ldots,1,-1)\in \R^{n \times n}$. It could be 
replaced, throughout, by any orthogonal self-adjoint matrix with determinant $-1$. We use the shorthand $[n]$ for the set $\{1,2,\ldots,n\}$
and $\Ieven$ for the set of subsets of $[n]$ with even cardinality.

    \subsection{Outline}
    \label{sec:outline}
    The remainder of the paper is organized as follows. In Section~\ref{sec:app} we describe a problem in satellite attitude
    estimation that can be reformulated as a semidefinite program using the ideas in this paper. 
     Section~\ref{sec:properties}
     focuses on the symmetry properties of $\conv\,SO(n)$ and $\conv\,O(n)$, as well as certain convex polytopes
     that naturally arise when studying these convex bodies.
    With these preliminaries established, Section~\ref{sec:reps} outlines the main arguments required to establish the 
    correctness of the spectrahedral representations
    of $SO(n)^\circ$, $O(n)^\circ$, $\conv\,SO(n)$ and $\conv\,O(n)$. Details of some of the constructions
    required for these arguments are deferred to Appendix~\ref{app:clif}. Section~\ref{sec:lower}
    establishes lower bounds on the size of spectrahedral representations of $SO(n)^\circ$, $O(n)^\circ$, 
    $\conv\,SO(n)$ and $\conv\,O(n)$ as well as a lower bound on the size of equivariant \psdlifts{} of $\conv\,SO(n)$. 

    Many of the properties of the convex bodies of interest in this paper are summarized in Table~\ref{tab:data} which may 
    serve as a useful navigational aid when reading the paper.
\begin{table}
    \ra{1.5}
    \begin{tabular}{@{}clllll@{}}\toprule
        $S$ &   & $SO(n) $ & $O(n)$ \\\midrule
            & Definition & $\{X\in \R^{n\times n}: X^TX = I,\,\det(X)=1\}$&$\{X\in \R^{n\times n}: X^TX = I\}$\\[0.2cm]
        $S^\circ$&&$SO(n)^\circ$ & $O(n)^\circ =$ Nuclear norm ball &\\\midrule
  & Diagonal slice & Polar of parity polytope\hfill(Prop.~\ref{prop:diag}) & Cross-polytope\hfill(Prop.~\ref{prop:diag})\\
                 &Spectrahedral & Size: $2^{n-1}$\hfill(Thm~\ref{thm:sonpolar}) & Size: $2^n$\hfill (Thm~\ref{thm:onpolar})\\[-0.2cm]
                     & representation & Matching lower bound\hfill(Thm~\ref{thm:lower}) & Matching lower bound\hfill(Thm~\ref{thm:lower}) \\
                &\psdlift{} & Size: $2^{n-1}$ & Size: $2n$\hfill (Eg.~\ref{eg:onpolarproj}) \\
        ${S^{\circ}}^{\circ}=$ & & &\\[-0.2cm]
        $\conv\,S$ & &$\conv\,SO(n)$ & $\conv\,O(n)=$ Operator norm ball &\\\midrule
        & Diagonal slice & Parity Polytope \hfill(Prop.~\ref{prop:diag})& Hypercube\hfill(Prop.~\ref{prop:diag})\\
        &\parbox{0pt}{Spectrahedral\\representation} & Size: $\begin{cases} 2^{n-1} + 2n & n\geq 4\\4 & n=3\end{cases}$\hfill(Thm~\ref{thm:son}) & Size: $2n$\hfill (Prop.~\ref{prop:on})\\[-0.2cm]
            & & Matching lower bound\hfill(Thm~\ref{thm:lower}) & Matching lower bound \hfill(Thm~\ref{thm:lower})\\ 
      & \psdlift{} &  Size: $ 2^{n-1}$\hfill (Cor~\ref{cor:sonproj}) & Size: $2n$\\
        \bottomrule
    \end{tabular}
    \caption{\label{tab:data} Summary of results related to the convex bodies considered in the paper.}
\end{table}

\section{An illustrative application---joint satellite attitude and spin-rate estimation}
\label{sec:app}

In this section we discuss a problem in satellite attitude estimation that can
be reformulated as semidefinite programs using the representation of
$SO(n)^\circ$ described in Section~\ref{sec:results}. Our aim, here, is to give
a concrete example of situations where the semidefinite representations we
describe in this paper arise naturally.  The problem of interest is one of
estimating the attitude (i.e.~orientation) and spin-rate of a spinning
satellite, and is a slight generalization of a problem posed recently by Psiaki
\cite{psiaki2009generalized}. We first focus on describing  the basic attitude
estimation problem in Section~\ref{sec:attitude}, before describing the joint
attitude and spin-rate estimation problem in Section~\ref{sec:joint}.

\subsection{Attitude estimation}
\label{sec:attitude}
The \emph{attitude} of a satellite is the element of $SO(3)$ that transforms a
reference coordinate system (the \emph{inertial system}) in which, say, the sun
is fixed, into a local coordinate system fixed with respect to the satellite's
body (the \emph{body system}).  We are given unit vectors $x_1,x_2,\ldots,x_T$
(e.g., the alignment of the Earth's magnetic field, directions of landmarks
such as the sun or other stars, etc.) in the inertial coordinate system, and
noisy measurements $y_1,y_2,\ldots,y_T$ of these directions in the body
coordinate system. Let $Q\in SO(3)$ denote the unknown attitude of the
satellite. The aim is to estimate (in the maximum likelihood sense) $Q$ given
the $y_k$, the $x_k$ and a description of the measurement noise. 

The simplest noise model assumes that each $y_k$ is independent has a von
Mises-Fisher distribution \cite{mardia2009directional} (a natural family of
probability distributions on the sphere) with mean $Qx_k$ and concentration
parameter $\kappa$ i.e.\ its probability density function is, up to a
proportionality constant that does not depend on $Q$,  $p(y_k;Q) \propto
\exp\left(\kappa\langle y_k,Qx_k\rangle\right)$.
Then the maximum likelihood estimate of $Q$ is found by solving 
\begin{equation}
    \label{eq:wahba}
    \max_{Q\in SO(3)}\sum_{k=1}^{T} \kappa \langle y_k,Qx_k\rangle = \max_{Q\in SO(3)} \langle Q,\kappa\sum_{k=1}^{T}y_kx_k^T\rangle = \max_{Q\in \conv\,SO(3)}
    \langle Q,\kappa\sum_{k=1}^{T}y_kx_k^T\rangle.
\end{equation}
This is a probabilistic interpretation of a problem known as \emph{Wahba's
problem} in the astronautical literature, posed by Grace Wahba in the July 1965
\emph{SIAM Review} problems and solutions section \cite[Problem 65-1]{wahba1965least}.

Our spectrahedral representation of $\conv\,SO(n)$  allows us to express the
optimization problem in~\eqref{eq:wahba} as a semidefinite program. In the
astronautical literature it is common to solve this problem via the 
$q$-method \cite{keat1977analysis} which involves parameterizing $SO(3)$ in
terms of unit quaternions and solving a symmetric eigenvalue problem. Our
semidefinite programming-based formulation could be thought of as a much more
flexible generalization of this eigenvalue problem-based approach that works
for any $n$, not just the case $n=3$. 

\subsection{Joint attitude and spin-rate estimation}
\label{sec:joint}

A significant benefit of having a semidefinite programming-based description of
a problem (such as Wahba's problem), is that it often allows us to devise
semidefinite programming-based solutions to more complicated related problems
by composing semidefinite representations in different ways.  An example of
this is given by the following generalization of Wahba's problem posed by 
Psiaki~\cite{psiaki2009generalized}.\footnote{Psiaki's formulation only
considers the $\kappa_2=0$ case, where measurements of the spin rate are not
considered.}

Consider a satellite rotating at a constant unknown angular velocity $\omega$
rad/sample around a known axis (e.g.~its major axis).  Assume the body
coordinate system is chosen so that the rotation is around the axis defined by
the first coordinate direction. Then the attitude matrix at the $k$th sample
instant is of the form
\[ Q(k) = \begin{bmatrix} 1 & 0 & 0\\0 & \cos(k \omega) & -\sin(k \omega)\\ 
0 & \sin(k \omega) & \cos(k \omega)\end{bmatrix}Q\]
where $Q\in SO(3)$ is the initial attitude.  Suppose, now, the satellite
\emph{sequentially} obtains measurements $y_0,y_1,\ldots,y_T$ in the body
coordinate system of known landmarks in the directions $x_0,x_1,\ldots,x_T$ in
the inertial coordinate system.  As before assume that the $y_k$ are
independent and have von Mises-Fisher distribution with mean $Q(k)x_k$ and
concentration parameter $\kappa_1$. Furthermore, the satellite obtains a
sequence $\omega_1,\omega_2,\ldots,\omega_T$ of noisy measurements of the
unknown constant spin rate $\omega$. Suppose the $\omega_k$ are independent and
each $\omega_k$ has a von Mises distribution \cite{mardia2009directional} (a
natural distribution for angular-valued quantities) with mean $\omega$ and
concentration parameter $\kappa_2$, i.e., its probability density function (up
to a constant independent of $\omega$) is given by $p(\omega_k;\omega) \propto
\exp\left(\kappa_2\cos(\omega_k - \omega)\right)$.  If the $\omega_k$ and the
$y_k$ are independent then the maximum likelihood estimate of $Q$ and $\omega$
can be found by solving
\begin{equation}
    \label{eq:psiaki}\max_{\substack{Q\in SO(3)\\\omega\in [0,2\pi)}} \sum_{k=0}^{T}\left\langle y_k,\kappa_1
        \left[\begin{smallmatrix} 1 & 0 & 0\\0 & \cos(k\omega) & -\sin(k\omega)\\
        0& \sin(k\omega ) & \cos(k\omega )\end{smallmatrix}\right]
        Qx_k\right\rangle + \kappa_2\sum_{k=0}^T \cos(\omega_k - \omega).
\end{equation}

Note that the optimization problem~\eqref{eq:psiaki} can be rewritten as
\begin{equation}
    \label{eq:psiaki2}
    \max_{\substack{Q\in SO(3)\\\omega\in[0,2\pi)}} a_1\cos(\omega) + b_1\sin(\omega) + \langle A_0,Q\rangle + 
        \sum_{k=1}^{T}\langle A_k,\cos(k\omega)Q\rangle + \langle B_k,\sin(k\omega)Q\rangle,
        \end{equation}
        i.e.~the maximization of a linear functional over
        \[ \mathcal{M}_{3,T} =
            \left\{(\cos(\omega),\sin(\omega),Q,\cos(\omega)Q,\sin(\omega)Q,\ldots,\cos(T\omega)Q,\sin(T\omega)Q):
            Q\in SO(3),\;\omega\in [0,2\pi)\right\}.\]
            We can reformulate this as a semidefinite program if we have a
            \psdlift{} of $\conv(\mathcal{M}_{3,T})$, because the optimization
            problem~\eqref{eq:psiaki2} is equivalent to the maximization of the
            same linear functional over $\conv(\mathcal{M}_{3,T})$.  Using the
            fact that $SO(n)^\circ$ has a spectrahedral representation of size
            $2^{n-1}$, it can be shown that that $\conv(\mathcal{M}_{n,T})$ has
            a \psdlift{} of size $2^{n-1}(T+1)$. Describing this in detail is
            beyond the scope of the present paper.  Instead we discuss this
            reformulation in further detail in a separate
            report~\cite{saunderson2014}.

\section{Basic properties of $\conv\,SO(n)$ and $\conv\,O(n)$}
\label{sec:properties}

In this section we consider the convex bodies $\conv\,SO(n)$ and $\conv\,O(n)$
purely from the point of view of convex geometry leaving discussion of aspects
related to their semidefinite representations for Section~\ref{sec:reps}. In
this section we describe their symmetries, and how the full space of
$\R^{n\times n}$ matrices decomposes with respect to these symmetries, via the
(special) singular value decomposition.  To a large extent one can characterize
$\conv\,SO(n)$ and $\conv\,O(n)$ in terms of their intersections with the
subspace of diagonal matrices.  These diagonal sections are well known
polytopes---the parity polytope and the hypercube respectively. The properties
of these diagonal sections are crucial to establishing our spectrahedral
representation of $\conv\,SO(n)$ in Section~\ref{sec:son} and the lower bounds
on the size of spectrahedral representations given in Section~\ref{sec:lower}. 

All of the results in this section are (sometimes implicitly) in the literature
in various forms. Here we aim for a brief yet unified presentation to make the
paper as self-contained as possible. 

\subsection{Symmetry and the special singular value decomposition}
\label{sec:symmetry}
In this section we describe the symmetries of $\conv\,O(n)$ and $\conv\,SO(n)$.

The group $O(n)\times O(n)$ acts on $\R^{n\times n}$ by $(U,V)\cdot X = UXV^T$.
This action leaves the set $O(n)$ invariant, and hence leaves the convex bodies
$\conv\,O(n)$ and $O(n)^\circ$ invariant. It is also useful to understand how
the ambient space of $n\times n$ matrices decomposes under this group action.
Indeed by the well-known singular value decomposition every element $X\in
\R^{n\times n}$ can be expressed as $X = U\Sigma V^T = (U,V)\cdot \Sigma$ where
$(U,V)\in O(n)\times O(n)$, and $\Sigma$ is diagonal with $\Sigma_{11} \geq
\cdots \geq \Sigma_{nn} \geq 0$. These diagonal elements are the \emph{singular
values}.  We denote them by $\sigma_i(X) = \Sigma_{ii}$. Note that for most of
what follows, we only use the fact that $\Sigma$ is diagonal, not that its
elements can be taken to be non-negative and sorted.

Similarly the group
\[ S(O(n)\times O(n)) = \{(U,V): U,V\in O(n),\; \det(U)\det(V) = 1\}\]
acts on $\R^{n\times n}$ by $(U,V)\cdot X = UXV^T$. This action leaves the sets
$SO(n)$ and $SO^-(n)$ invariant, and hence leaves the convex bodies
$\conv\,SO(n)$, $\conv\,SO^-(n)$, $SO(n)^\circ$, $SO^-(n)^\circ$, $\conv\,O(n)$
and $O(n)^\circ$ invariant. A variant on the singular value decomposition,
known as the \emph{special singular value
decomposition}~\cite{sanyal2011orbitopes} describes how the space of $n\times
n$ matrices decomposes under this group action. Indeed every $X\in \R^{n\times
n}$ can be expressed as $X = U\tilde{\Sigma}V^T = (U,V)\cdot \tilde{\Sigma}$
where $(U,V)\in S(O(n)\times O(n))$ and $\tilde{\Sigma}$ is diagonal with
$\tilde{\Sigma}_{11} \geq \cdots \geq \tilde{\Sigma}_{n-1,n-1} \geq
|\tilde{\Sigma}_{nn}|$. These diagonal elements are the \emph{special singular
values}. We denote them by $\tilde{\sigma}_i(X) = \tilde{\Sigma}_{ii}$. Again
in what follows we typically only use the fact that $\tilde{\Sigma}$ is
diagonal for our arguments. 

The special singular value decomposition can be obtained from the singular
value decomposition.  Suppose $X = U\Sigma V^T$ is a singular value
decomposition of $X$ so that $(U,V)\in O(n)\times O(n)$.  If $\det(U)\det(V) =
1$ this is also a valid special singular value decomposition. Otherwise, if
$\det(U)\det(V) = -1$ then $X = UR(R\Sigma)V^T$ gives a decomposition where
$(UR,V)\in S(O(n)\times O(n))$ and $R\Sigma$ is again diagonal, but with the
last diagonal entry being negative. As such the singular values and special
singular values of an $n\times n$ matrix are related by $\sigma_i(X) =
\tilde{\sigma}_i(X)$ for $i=1,2,\ldots,n-1$ and $\tilde{\sigma}_n(X) =
\textup{sign}(\det(X))\sigma_n(X)$. 

The importance of these decompositions of $\R^{n\times n}$ under the action of
$O(n)\times O(n)$ and $S(O(n)\times O(n))$ is that they allow us to reduce many
arguments, by invariance properties, to arguments about diagonal matrices. 

\subsection{Polytopes associated with $\conv\,O(n)$ and $\conv\,SO(n)$}
The convex hull of $O(n)$ is closely related to the \emph{hypercube} 
\begin{equation}
    \label{eq:cube1}
    \Cube_n =  \conv\{x\in \R^n: x_i^2 = 1,\;\;\text{for $i\in [n]$}\};
\end{equation}
the convex hull of $SO(n)$ is closely related to the \emph{parity polytope} 
\begin{equation}
    \label{eq:pp1}
    \PP_n = \conv\{x\in \R^n: \textstyle{\prod_{i=1}^{n}x_i = 1},\;\;x_i^2 = 1,\;\;\text{for $i\in [n]$}\};
\end{equation}
the convex hull of $SO^-(n)$ is closely related to the \emph{odd parity polytope}
\begin{equation}
    \label{eq:ppminus1}
    \PP_n^- = \conv\{x\in \R^n: \textstyle{\prod_{i=1}^{n}x_i = -1}.\;\;x_i^2 = 1,\;\;\text{for $i\in [n]$}\}.
\end{equation}
In this section we briefly discuss these  polytopes as well as showing that they are the diagonal
sections of $\conv\,O(n)$, $\conv\,SO(n)$ and $\conv\,SO^-(n)$ respectively.

\paragraph{Facet descriptions}
Irredundant descriptions of $\Cube_n$ and $\PP_n$ in terms of linear
inequalities are well known \cite{jeroslow1975defining}.  The hypercube has
$2n$ facets corresponding to the linear inequality description
\begin{equation}
    \label{eq:cubef}
    \Cube_n = \{x\in \R^n: -1 \leq x_i \leq 1\;\text{for $i\in [n]$} \}.
\end{equation}
For $n\geq 4$ the parity polytope $\PP_n$ has $2n+2^{n-1}$ facets corresponding
to the linear inequality description
\begin{equation}
    \label{eq:ppnf}
     \PP_n = \bigg\{x\in \R^n: -1\leq x_i \leq 1\;\text{for $i\in [n]$},\;\sum_{i\notin I}x_i - \sum_{i\in I}x_i \leq n-2\;\;\text{for $I\subseteq [n]$, $|I|$ odd}\bigg\}.
\end{equation}
In the cases $n=2$ and $n=3$ this description simplifies to
\begin{align}
    \PP_2 & = \left\{\left[\begin{smallmatrix} x\\x\end{smallmatrix}\right]\in \R^2: -1\leq x \leq 1\right\}\label{eq:pp2f}\\
    \PP_3 & =\left\{x\in \R^3: x_1-x_2+x_3\leq 1,\;-x_1+x_2+x_3\leq 1,\right.\nonumber\\
          &\qquad\qquad\quad\;\,\left.x_1+x_2-x_3\leq 1,\;-x_1-x_2-x_3\leq 1\right\}\label{eq:pp3f}
\end{align}
showing that $\PP_3$ has only four facets.

The polar of the hypercube is the \emph{cross-polytope}. We denote it by
$\Cube_n^\circ$. It is clear from \eqref{eq:cube1} that $\Cube_n^\circ$ has
$2^n$ facets and corresponding linear inequality description
\begin{equation}
    \label{eq:cubepolarf}
    \Cube_n^\circ = \bigg\{x\in \R^n: \sum_{i\notin I}x_i - \sum_{i\in I}x_i \leq 1 \quad\text{for $I\subset [n]$}\bigg\}. 
\end{equation}
The polar of the parity polytope is denoted by $\PP_n^\circ$. It is clear from
\eqref{eq:pp1} that $\PP_n^\circ$ has $2^{n-1}$ facets and corresponding linear
inequality description
\begin{equation}
    \label{eq:pppolarf}
    \PP_n^\circ = \bigg\{x\in \R^n: \sum_{i\notin I} x_i - \sum_{i\in I} x_i \leq 1\quad\text{for $I\subset [n]$, $|I|$ even}\bigg\}.
\end{equation}
Similarly 
\begin{equation}
    \label{eq:ppnegpolarf}
    {\PP_n^-}^\circ = \bigg\{x\in \R^n: \sum_{i\notin I} x_i - \sum_{i\in I}x_i \leq 1 \quad\text{for $I\subseteq [n]$, $|I|$ odd}\bigg\}.
    \end{equation}
To get a sense of the importance of these polytopes for understanding
$\conv\,SO(n)$  it may be instructive to compare  \eqref{eq:pp2f} with
\eqref{eq:so2}, \eqref{eq:pp3f} with \eqref{eq:so3}, \eqref{eq:ppnf} with
\eqref{eq:son}, and \eqref{eq:pppolarf} with \eqref{eq:sonpolar}.

We conclude the discussion of these polytopes with a useful alternative
description of $\PP_n$.

\begin{lemma}
    \label{lem:ppalt}
    The parity polytope can be expressed as 
    \[ \PP_n = \Cube_n \cap (n-2) \cdot {\PP_n^-}^\circ.\]
    In the case $n=3$ this simplifies to $\PP_3 = {\PP_3^-}^\circ$.
\end{lemma}
\begin{proof}
For the general case, we need only examine the facet descriptions
in~\eqref{eq:cubef},~\eqref{eq:ppnf}, and~\eqref{eq:ppnegpolarf}.  In the case
$n=3$ the result follows by comparing \eqref{eq:pp3f} with
\eqref{eq:ppnegpolarf}.
\end{proof}

\paragraph{Diagonal projections and sections}
We now establish the link between the hypercube and the convex hull of $O(n)$,
and the parity polytope and the convex hull of $SO(n)$. First we prove a result
that says that the subspace $\Diag$ of diagonal matrices interacts particularly
well with these convex bodies. The theorem applies for the convex bodies
$\conv\,O(n)$, $\conv\,SO(n)$ and $\conv\,SO^-(n)$ because whenever $g$ is a
diagonal matrix with non-zero entries in $\{-1,1\}$ (a diagonal sign matrix)
then each of these convex bodies is invariant under the conjugation map $X
\mapsto gXg^T$. 
\begin{lemma}
    \label{lem:symfix}
    Let $C\subset \R^{n\times n}$ be a convex body that is invariant under conjugation by diagonal sign matrices. 
    Then $\projto{\Diag}(C) = \projto{\Diag}(C\cap \Diag)$ and 
    $[\projto{\Diag}(C\cap \Diag)]^\circ = \projto{\Diag}(C^\circ \cap \Diag)$.
\end{lemma}
\begin{proof}
    We first establish that $\projto{\Diag}(C) = \projto{\Diag}(C\cap \Diag)$.
    Note that clearly $\projto{\Diag}(C\cap \Diag)\subseteq \projto{\Diag}(C)$.
    For the reverse inclusion let $G$ denote the group of diagonal sign
    matrices  and observe that $\Diag$ is the subspace of $n\times n$ matrices
    fixed pointwise by the conjugation action of diagonal sign matrices. Then
    for any $X\in C$ the projection onto $\Diag$, the fixed point subspace, is
    \[ \projtofrom{\Diag}(X) =  \frac{1}{2^n}\sum_{g\in G}gXg^T\] which gives a
    description of $\projtofrom{\Diag}(X)$ as a convex combination of the
    $gXg^T$. Each $gXg^T\in C$ since  $C$ is invariant under conjugation by
    diagonal sign matrices. Hence $\projtofrom{\Diag}(X)\in C\cap \Diag$ and so
    $\projto{\Diag}(X)\in \projto{\Diag}(C\cap \Diag)$. 

    Now we  establish that $[\projto{\Diag}(C\cap \Diag)]^\circ =
    \projto{\Diag}(C^\circ \cap \Diag)$.  For any $y\in \Diag$ we have that 
    \[ \max_{x\in \projto{\Diag}(C\cap \Diag)} \langle y,x\rangle = \max_{x\in \projto{\Diag}(C)}\langle y,x\rangle 
        = \max_{z\in C}\langle y,\projto{\Diag}(z)\rangle = \max_{z\in C}\langle \projfrom{\Diag}(y),z\rangle.\]
        Hence $y\in [\projto{\Diag}(C\cap \Diag)]^\circ$ if and only if $\projfrom{\Diag}(y)\in C^\circ$, or, equivalently,
        $y\in \projto{\Diag}(C^\circ \cap\Diag)$.
    \end{proof}
    We note that this lemma generalizes to the situation where $C$ is a convex
    body invariant under the action of a compact group and the subspace $\Diag$
    is replaced with the fixed point subspace of the group action. 

    The key fact that relates the parity polytope and the convex hull of
    $SO(n)$ is the following celebrated theorem of Horn~\cite{horn1954doubly}.
\begin{theorem}[Horn]
    \label{thm:horn}
    The projection onto the diagonal of $SO(n)$ is the parity polytope, i.e.\
    $\projto{\Diag}(SO(n)) = \PP_n$. 
\end{theorem}
Note that we do not need the full strength of Horn's theorem. We only use the
corollaries that 
\begin{align}
    \projto{\Diag}(\conv\,SO(n)) & = \conv\,\projto{\Diag}(SO(n)) = \conv\,\PP_n = \PP_n\quad\text{and}\label{eq:horn1}\\
    \projto{\Diag}(\conv\,SO^-(n)) & = \projto{\Diag}(R\cdot \conv\,SO(n)) = R\cdot\projto{\Diag}(\conv\,SO(n)) = R\cdot \PP_n = \PP_n^-\label{eq:horn2}.
\end{align}
We are now in a position to establish the main result of this section. 
\begin{proposition}
    \label{prop:diag}
    Let $\Diag\subset \R^{n\times n}$ denote the subspace of diagonal matrices. Then
    \begin{align}
        \projto{\Diag}(\Diag \cap \conv\,O(n)) & = \Cube_n, \;\;\;\,\;\;\;\;\;\;\projto{\Diag}(\Diag \cap O(n)^\circ) = \Cube_n^\circ,\nonumber\\
        \projto{\Diag}(\Diag \cap \conv\,SO(n)) & = \PP_n, \;\;\;\;\,\projto{\Diag}(\Diag\cap SO(n)^\circ) = \PP_n^\circ,\nonumber\\
        \projto{\Diag}(\Diag \cap \conv\,SO^-(n)) & = \PP_n^-, \;\;\projto{\Diag}(\Diag\cap SO^-(n)^\circ) = {\PP_n^-}^\circ.\nonumber
    \end{align}
\end{proposition}
\begin{proof}
    First note that by~\eqref{eq:horn1} and~\eqref{eq:horn2} we know that
    $\projto{\Diag}(\conv\,SO(n)) = \PP_n$ and that
    $\projto{\Diag}(\conv\,SO^-(n)) = \PP_n^-$. Consequently 
    \[ \projto{\Diag}(\conv\,O(n)) = \conv\,\projto{\Diag}(SO(n)\cup SO^-(n)) = \conv\,(\PP_n \cup \PP_n^-) = \Cube_n.\]   

    Since each of $\conv\,O(n)$, $\conv\,SO(n)$, $\conv\,SO^-(n)$ is invariant
    under conjugation by diagonal sign matrices we can apply
    Lemma~\ref{lem:symfix}. Doing so and using the characterization of the
    diagonal projections of each of these convex bodies from the previous
    paragraph, completes the proof.
\end{proof}

\section{Spectrahedral representations of $SO(n)^\circ$ and $\conv\,SO(n)$}
\label{sec:reps}
This section is devoted to outlining the proofs of
Theorems~\ref{thm:sonpolar},~\ref{thm:onpolar} and~\ref{thm:son}, giving
spectrahedral representations of $SO(n)^\circ$ $O(n)^\circ$ and $\conv\,SO(n)$.
For the sake of exposition, we initially focus on $SO(2)^\circ$ as in this case
all the ideas are familiar.  Low dimensional coincidences do mean that some
issues are simpler in the $2\times 2$ case than in general.  After discussing
the $2\times 2$ case, in Section~\ref{sec:gen} we generalize the argument,
deferring some details to Appendix~\ref{app:clif}. Finally in
Section~\ref{sec:son} we construct our spectrahedral representation of
$\conv\,SO(n)$.

\subsection{The $2\times 2$ case}

We begin by giving a spectrahedral representations of $SO(2)^{\circ}$.  
We make crucial use of the trigonometric identities
$\cos(\theta) = \cos^2(\theta/2) - \sin^2(\theta/2)$ and $\sin(\theta) =
2\cos(\theta/2)\sin(\theta/2)$.  Recall that elements of $SO(2)$ have the form 
\[ \begin{bmatrix} \cos(\theta) & -\sin(\theta)\\\sin(\theta) & \cos(\theta)\end{bmatrix} = 
\begin{bmatrix} 
    \cos^2(\textstyle{\frac{\theta}{2}}) - \sin^{2}(\textstyle{\frac{\theta}{2}}) & 
-2\cos(\textstyle{\frac{\theta}{2}})\sin(\textstyle{\frac{\theta}{2}})\\
2\cos(\textstyle{\frac{\theta}{2}})\sin(\textstyle{\frac{\theta}{2}}) & 
\cos^2(\textstyle{\frac{\theta}{2}}) - \sin^{2}(\textstyle{\frac{\theta}{2}})\end{bmatrix}
\]
and that $(\cos(\theta/2),\sin(\theta/2))$ parameterizes the unit circle in
$\R^2$. Hence $SO(2)$ is the image of the unit
circle $\{(x_1,x_2): x_1^2+x_2^2=1\}$ under the quadratic map
\[ Q(x_1,x_2) = \begin{bmatrix} x_1^2 -x_2^2 & -2x_1x_2\\2x_1x_2 & x_1^2 - x_2^2\end{bmatrix}.\]
As such, $Y\in SO(2)^{\circ}$ if and only if, for all $(x_1,x_2)$ in the unit
circle,
\begin{align*}
    \langle Y,Q(x_1,x_2)\rangle & = 
\left\langle \begin{bmatrix} Y_{11} & Y_{12}\\Y_{21} & Y_{22}\end{bmatrix},
\begin{bmatrix} x_1^2 - x_2^2 & -2x_1x_2\\2x_1x_2 & x_1^2 - x_2^2\end{bmatrix}\right\rangle\\  
& = \begin{matrix}\begin{bmatrix}x_1 & x_2\end{bmatrix}\\~\end{matrix}
\begin{bmatrix} Y_{11} + Y_{22} & Y_{21} - Y_{12}\\Y_{21} - Y_{12} & -Y_{11} - Y_{22}\end{bmatrix}
\begin{bmatrix} x_1\\x_2\end{bmatrix} \leq 1.
\end{align*}
This is equivalent to the spectrahedral representation
\[ SO(2)^\circ = \left\{Y: \begin{bmatrix} Y_{11} + Y_{22} & Y_{21} -
Y_{12}\\Y_{21} - Y_{12} &-Y_{11} - Y_{22}\end{bmatrix} \nsd I \right\}\]
which coincides with the $n=2$ case of Theorem~\ref{thm:sonpolar}.

    To summarize, the main idea of the argument is that we use a
    parameterization of $SO(2)$ as the image of the unit circle under a
    quadratic map. This parameterization allows us to rewrite the maximum of a
    linear functional on $SO(2)$ as the maximum of a quadratic form on the unit
    circle which can be expressed as a spectrahedral condition.

    We note that a very similar argument works in the case $n=3$ to directly
    produce the representations of $SO(3)^\circ$ and $\conv\,SO(3)$ in
    Theorem~\ref{thm:sonpolar} and Corollary~\ref{cor:sonproj} respectively.
    Indeed the unit quaternion parameterization of rotations gives a
    parameterization of $SO(3)$ as the image of the unit sphere in $\R^4$ under
    a quadratic mapping again allowing us to rewrite the maximum of a linear
    functional on $SO(3)$ as the maximum of a quadratic form on the unit sphere
    which is equivalent to a spectrahedral condition.
    
    \subsection{Outline of the general argument}
    \label{sec:gen}
    
    For the general case, we first need a quadratic parameterization of
    $SO(n)$.  There is a classical construction of a quadratic map
    $Q:\R^{2^{n-1}}\rightarrow \R^{n\times n}$ and a subset $\Spin(n)$ of the
    unit sphere in $\R^{2^{n-1}}$ such that $SO(n) = Q(\Spin(n))$.  (We recall
    this construction in Appendix~\ref{app:clif}, only discussing those aspects
    relevant for our argument here.)
    
    Unfortunately, for $n\geq 4$, $\Spin(n)$ is a \emph{strict} subset of the
    unit sphere in $\R^{2^{n-1}}$, so we cannot simply follow the argument for
    the $n=2$ case verbatim. The key difficulty is that we need a spectrahedral
    characterization of the maximum \emph{over $\Spin(n)$} of the quadratic
    form $x\mapsto \langle Y,Q(x)\rangle$ (for arbitrary $Y$). It is not
    obvious how to do this when $\Spin(n)$ is a strict subset of the sphere.

    We achieve this by showing that for any $Y$, the maximum of the quadratic
    form $x\mapsto \langle Y,Q(x)\rangle$ over the entire sphere coincides with
    its maximum over the strict subset $\Spin(n)$ of the sphere (see
    Proposition~\ref{prop:eig}, to follow).  To establish this we exploit
    additional structure in $\Spin(n)$ and certain equivariance properties of
    $Q$.  The specific properties we use are stated in
    Propositions~\ref{prop:spin},~\ref{prop:basis},~\ref{prop:mef},
    and~\ref{prop:spinequivariance}. We prove these in Appendix~\ref{app:clif}. 
    \begin{proposition}
        \label{prop:spin}
        There is a $2^{n-1}$-dimensional inner product space, $\Cl^0(n)$, 
        a subset $\Spin(n)$ of the unit sphere in $\Cl^0(n)$ and a quadratic map 
        $Q:\Cl^0(n)\rightarrow \R^{n\times n}$ such that $Q(\Spin(n)) = SO(n)$. 
    \end{proposition}
    Given an orthonormal basis $e_1,\ldots,e_n$ for $\R^n$ there is a
    corresponding orthonormal basis $(e_I)_{I\in \Ieven}$ for $\Cl^0(n)$
    indexed by $\Ieven$, the subsets of $\{1,2,\ldots,n\}$ of even cardinality.
    This basis has the following important property.
    \begin{proposition}
        \label{prop:basis}
        Each of the $2^{n-1}$ elements of the basis $(e_I)_{I\in \Ieven}$ is in $\Spin(n)$.
    \end{proposition}
    For the rest of this section we fix these two choices of basis for $\R^n$
    and $\Cl^0(n)$ respectively. With respect to the basis $(e_{I})_{I\in
    \Ieven}$ we write $x\in \Cl^0(n)$ in coordinates as $x = \sum_{I\in \Ieven}
    x_I e_I$.  The following result summarizes the properties of $Q(x)$, with
    respect to these choices of basis, that we use.  
   \begin{proposition}
       \label{prop:mef}
    The matrix entry functions $[Q(\cdot)]_{ij}:\Cl^0(n)\rightarrow \R$ are
    explicitly given by
    \[ [Q(x)]_{ij} = \langle x,A_{ij}x\rangle = \sum_{I,J\in \Ieven}[A_{ij}]_{I,J} x_I x_J\]
    where the $A_{ij}$ are the signed permutation matrices defined
    in~\eqref{eq:AB}.  In particular, the quadratic forms $[Q(x)]_{ii}$ are
    diagonal with respect to the basis $(e_I)_{I\in \Ieven}$. 
    \end{proposition}
    Finally, $Q$ interacts well with left and right multiplication by elements of $SO(n)$. 
    \begin{proposition}
        \label{prop:spinequivariance}
        For any $U,V\in SO(n)$ there is a linear map
        $\Phi_{(U,V)}:\Cl^0(n)\rightarrow \Cl^0(n)$ such that
        \begin{itemize}
            \item $UQ(x)V^T = Q(\Phi_{(U,V)}x)$ for all $x\in \Cl^0(n)$
            \item $\Phi_{(U,V)}$ is invertible
            \item $\Phi_{(U,V)}$ and $\Phi_{(U,V)}^{-1}$ preserve $\Spin(n)$, i.e., $\Phi_{(U,V)}(\Spin(n)) = \Phi_{(U,V)}^{-1}(\Spin(n)) = \Spin(n)$. 
        \end{itemize}
\end{proposition}
    
    The following proposition, the crux of our argument, implies that for any
    $n\times n$ matrix $Y$, the maximum of the quadratic form $x\mapsto \langle
    Y,Q(x)\rangle$ quadratic form over the whole sphere and over the (strict)
    subset $\Spin(n)$, coincide. 
    \begin{proposition}
        \label{prop:eig}
        Given any $Y\in \R^{n\times n}$ the quadratic form $x\mapsto \langle
        Y,Q(x)\rangle$ has a basis of eigenvectors that are elements of
        $\Spin(n)$.
    \end{proposition}
    \begin{proof}
         Suppose $Y\in \R^{n\times n}$ is arbitrary. Then by the special
         singular value decomposition $Y$ can be expressed as $Y = U^TDV$ where
         $U$ and $V$ are in $SO(n)$ and $D$ is diagonal.  Then by
         Proposition~\ref{prop:spinequivariance}
         \[ \langle Y,Q(x)\rangle = \langle U^TDV,Q(x)\rangle = \langle D,UQ(x)V^T\rangle = \langle D, Q(\Phi_{(U,V)}x)\rangle.\]
         Consider the quadratic form $z\mapsto \langle D,Q(z)\rangle$. Observe
         that 
         \[ \langle D,Q(z)\rangle = \sum_{i=1}^{n}D_{ii}[Q(z)]_{ii} = \langle z,\left[\textstyle{\sum_{i=1}^{n}}D_{ii} A_{ii}\right]z\rangle\]
         and (by Proposition~\ref{prop:mef}) each of the $A_{ii}$ is diagonal.
         Hence $\sum_{i=1}^{n}D_{ii}A_{ii}$ is diagonal and so $z\mapsto
         \langle D,Q(z)\rangle$ has $(e_I)_{I\in \Ieven}$ as a basis of
         eigenvectors. Hence the quadratic form $x\mapsto \langle
         Y,Q(x)\rangle$ has $\Phi_{(U,V)}^{-1}e_I$ for $I\in \Ieven$ as a basis
         of eigenvectors.  Since the $e_I$ are in $\Spin(n)$ (by
         Proposition~\ref{prop:basis}), $\Phi_{(U,V)}$ is invertible, and
         $\Phi_{(U,V)}^{-1}$ preserves $\Spin(n)$ (by
         Proposition~\ref{prop:spinequivariance}) we can conclude that the
         quadratic form $x\mapsto \langle Y,Q(x)\rangle$ has a basis of
         eigenvectors all of which are elements of $\Spin(n)$. 
     \end{proof}
    Assuming Propositions~\ref{prop:spin} and~\ref{prop:eig} we can prove
    Theorem~\ref{thm:sonpolar} using an embellishment of the same argument we
    used in the $2\times 2$ case.
    \begin{proof}[Proof of Theorem~\ref{thm:sonpolar}]
        Since the image of $\Spin(n)$ under $Q$ is $SO(n)$, an $n\times n$
        matrix $Y$ is in $SO(n)^\circ$ if and only if
        \[\max_{X\in SO(n)}\langle Y,X\rangle = \max_{x\in \Spin(n)}\langle Y,Q(x)\rangle \leq 1.\]
        Since $\Spin(n)$ is a subset of the unit sphere in $\Cl^0(n)$, we have
        that 
        \[ \max_{x\in \Spin(n)}\langle Y,Q(x)\rangle \leq \max_{\substack{x\in \Cl^0(n)\\ \langle x,x\rangle = 1}} \langle Y,Q(x)\rangle.\]
        The maximum of the quadratic form $x\mapsto \langle Y,Q(x)\rangle$ over
        the unit sphere in $\Cl^0(n)$ occurs at any eigenvector corresponding
        to the largest eigenvalue of the quadratic form. By
        Proposition~\ref{prop:eig} we can always find such an eigenvector in
        $\Spin(n)$, establishing that 
        \[ \max_{x\in \Spin(n)}\langle Y,Q(x)\rangle = \max_{\substack{x\in \Cl^0(n)\\ \langle x,x\rangle = 1}} \langle Y,Q(x)\rangle.\]
        Hence $Y\in SO(n)^\circ$ if and only if for all $x\in \Cl^0(n)$ such
        that $\langle x,x\rangle = 1$,  
        \[ \langle Y,Q(x)\rangle = \left\langle x,\textstyle{\sum_{i,j=1}^{n}}Y_{ij}A_{ij} x\right\rangle \leq 1.\]
        This is equivalent to the spectrahedral representation given in
        Theorem~\ref{thm:sonpolar}.
    \end{proof}
    \begin{remark}
        We briefly describe a more geometric dual interpretation of the
        arguments that establish Theorem~\ref{thm:sonpolar}.  Throughout this
        remark let $S = \{x\in \Cl^0(n): \langle x,x\rangle = 1\}$ be the unit
        sphere in $\Cl^0(n)$.  We have seen that there is a quadratic map $Q$
        such that $SO(n) = Q(\Spin(n)) \subset Q(S)$ with the inclusion being
        strict for $n\geq 4$.  The remainder of the proof of
        Theorem~\ref{thm:sonpolar} shows, from this viewpoint, that
        $\conv\,SO(n) = \conv\,Q(\Spin(n)) = \conv\,Q(S)$, i.e.\ all the points
        in $S$ that are not in $\Spin(n)$ are mapped by $Q$ inside the convex
        hull of $Q(\Spin(n))$.  One may wonder whether $Q(S) = \conv\,SO(n)$,
        i.e.~whether the image of the sphere under $Q$ is actually convex.
        This is not the case---already for $n=2$ we can see that $Q(S) = SO(2)
        \neq \conv\,SO(2)$. 
    \end{remark}

    It is now straightforward to prove Theorem~\ref{thm:onpolar}, giving a
    spectrahedral representation of $O(n)^\circ$ of size $2^n$.
    \begin{proof}[Proof of Theorem~\ref{thm:onpolar}]
        Since $O(n)^\circ = SO(n)^\circ \cap SO^-(n)^\circ$ (see
        \eqref{eq:onpolar-intersection}) and we have already constructed a
        spectrahedral representation of $SO(n)^\circ$, it remains to give a
        spectrahedral representation of $SO^-(n)^\circ$.  Since $SO^-(n) =
        R\cdot SO(n)$ where $R = \diag(1,1,\ldots,1,-1)$, it follows that $Y\in
        SO^-(n)^\circ$ if and only if $\langle Y,RX\rangle = \langle
        RY,X\rangle \leq 1$ for all $X\in SO(n)$.  Hence $Y\in SO^-(n)^\circ$
        if and only if $RY \in SO(n)^\circ$. 

        From these observations and Theorem~\ref{thm:sonpolar} we have that 
        \begin{equation}
            \label{eq:onpolar2}
            O(n)^\circ = SO(n)^\circ \cap SO^-(n)^\circ = \bigg\{Y\in \R^{n\times n}: \sum_{i,j=1}^{n}Y_{ij}A_{ij} \nsd I,\; 
        \sum_{i,j=1}^{n}[RY]_{ij}A_{ij} \nsd I\bigg\}
    \end{equation}
    which is a spectrahedral representation of size $2^n$.
    \end{proof}
    \subsection{A spectrahedral representation of $\conv\,SO(n)$}
    \label{sec:son}
    In this section we give a spectrahedral representation of $\conv\,SO(n)$
    using a description of $\conv\,SO(n)$ which is inherited from the
    corresponding description of the parity polytope.
    \begin{proposition}
        \label{prop:son2}
        The convex hull of $SO(n)$ can be expressed in terms of $\conv\,O(n)$ and $SO(n)^\circ$ as
    \[ \conv\,SO(n) = \conv\,O(n) \cap (n-2)  SO^-(n)^\circ.\]
    If $n=3$ this simplifies to 
    $\conv\,SO(3) = SO^-(3)^\circ$.
\end{proposition}
\begin{proof}
    Suppose $X\in \R^{n\times n}$ is arbitrary. By the special singular value
    decomposition $X = U\tilde{\Sigma}V^T$ where $(U,V)\in S(O(n)\times O(n))$
    and $\tilde{\Sigma} = \diag^*(\tilde{\sigma})$ is diagonal.  Then since
    $SO(n)$ is invariant under the action of $S(O(n)\times O(n))$, it follows
    that $X\in \conv\,SO(n)$ if and only if $\tilde{\Sigma}\in \conv\,SO(n)\cap
    \Diag$. Similarly since $\conv\,O(n)$ and $SO^-(n)^\circ$ are invariant
    under the action of $S(O(n)\times O(n))$, it follows that $X\in
    \conv\,O(n)\cap (n-2)SO^-(n)^\circ$ if and only if $\tilde{\Sigma}\in
    \conv\,O(n)\cap \Diag$ and $\tilde{\Sigma}\in (n-2)SO^-(n)^\circ \cap
    \Diag$. 
    
    Since the diagonal section of $\conv\,SO(n)$ is the parity polytope, $X\in
    \conv\,SO(n)$ if and only if $\tilde{\sigma}\in \PP_n$.  Since the diagonal
    section of $\conv\,O(n)$ is the hypercube, $\tilde{\sigma}\in \Cube_n$ if
    and only if $\tilde{\Sigma}\in \conv\,O(n)\cap \Diag$.  Since the diagonal
    section of $SO^-(n)^\circ$ is ${\PP_n^-}^\circ$, $\tilde{\sigma}\in(n-2)
    {\PP_n^-}^\circ$ if and only if $\tilde{\Sigma} \in (n-2)SO^-(n)^\circ\cap
    \Diag$. 
   
     Finally we use the fact that $\PP_n = \Cube_n \cap (n-2){\PP^-_n}^\circ$
     (see Lemma~\ref{lem:ppalt}). Then  $X\in \conv\,SO(n)$ if and only if
     $\tilde{\sigma}\in \PP_n$ which occurs if and only if $\tilde{\sigma}\in
     \Cube_n$ and $\tilde{\sigma}\in (n-2){\PP_n^-}^\circ$ which occurs if and
     only if $X\in \conv\,O(n)\cap (n-2)SO^-(n)^\circ$.

    In the case $n=3$ the description $\PP_n = \Cube_n \cap
    (n-2){\PP_n^-}^\circ$ simplifies to $\PP_3 = {\PP_3^-}^\circ$. The
    corresponding simplification propagates through the above argument to give
    $\conv\,SO(3) = SO^-(3)^\circ$.
\end{proof}
Since the description of $\conv\,SO(n)$ in Proposition~\ref{prop:son2} involves
$\conv\,O(n)$, we first give the well-known spectrahedral representation of
$\conv\,O(n)$.
    \begin{proposition}
        \label{prop:on}
        The convex hull of $O(n)$ is a spectrahedron. An explicit spectrahedral
        representation of size $2n$ is given by
\begin{equation}
\label{eq:onspec}
\conv\,O(n) = \left\{X\in \R^{n\times n}: \begin{bmatrix} 0 & X\\X^T & 0\end{bmatrix} \nsd I_{2n}\right\}.
\end{equation}
    \end{proposition}
    \begin{proof}
        Let $Q\in O(n)$ be arbitrary. Then since $Q^TQ = I_n$ it follows that  
    \[ \begin{bmatrix} I_n & -Q\\-Q^T & I_n\end{bmatrix} = 
        \begin{bmatrix} I_n\\-Q^T\end{bmatrix}\begin{matrix}\begin{bmatrix} I_n &
    -Q\end{bmatrix}\\\begin{matrix} & \end{matrix}\end{matrix}\psd 0\]
    and so $Q$ is an element of the right hand side of \eqref{eq:onspec}. Since
the right hand side of~\eqref{eq:onspec} is convex, it follows that
$\conv\,O(n) \subseteq \left\{X\in \R^{n\times n}: \left[\begin{smallmatrix} 0
& X\\X^T & 0\end{smallmatrix}\right] \nsd I_{2n}\right\}$.

    For the reverse inclusion, suppose $X$ is an element of the right hand side of~\eqref{eq:onspec}. By the singular value decomposition there is a 
    diagonal matrix $\Sigma$ such that $X = U\Sigma V^T$ where $U,V\in O(n)$. Conjugating by the orthogonal matrix
    $\left[\begin{smallmatrix} U^T & 0\\0 & V^T\end{smallmatrix}\right]$ we see that 
\[ \begin{bmatrix} 0 & X\\X^T & 0\end{bmatrix} \nsd I_{2n} \iff \begin{bmatrix} 0 & \Sigma\\\Sigma & 0\end{bmatrix} \nsd I_{2n}\]
    which is equivalent to $-1\leq \Sigma_{ii}\leq 1$ for $i\in [n]$. Since
    $\Diag\cap \conv\,O(n)$ is the hypercube it follows that $\Sigma \in \Diag
    \cap \conv\,O(n)$ and so that $U\Sigma V^T \in \conv\,O(n)$. 
\end{proof}

\begin{proof}[Proof of Theorem~\ref{thm:son}]
    Since we now have a spectrahedral representation of $\conv\,O(n)$ (from~\eqref{eq:onspec}) and of $SO^-(n)^\circ$ (from~\eqref{eq:onpolar2}), 
    we can use Proposition~\ref{prop:son2} to combine them to give the spectrahedral representation
\[ \conv\,SO(n) = \left\{X\in \R^{n\times n}: \begin{bmatrix} 0 & X\\X^T & 0\end{bmatrix} \nsd I_{2n},\;\;
\sum_{i,j=1}^{n}A_{ij}[RX]_{ij} \nsd (n-2)I_{2^{n-1}}\right\}.\]
In the case $n=3$ Proposition~\ref{prop:son2} tells us that $\conv\,SO(3) = SO^-(3)^\circ$ and so 
\[\conv\,SO(3) = \left\{X\in \R^{3\times 3}: \sum_{i,j=1}^{3}A_{ij}[RX]_{ij} \nsd I_{4}\right\}\]
            which can be expressed explicitly as in~\eqref{eq:so3} by using the definition of the $A_{ij}$ in~\eqref{eq:AB}. 

            To conclude the proof we explicitly simplify the spectrahedral representation~\eqref{eq:son} for the case $n=2$. Indeed
\[ \conv\,SO(2) = \bigg\{X\in \R^{2\times 2}: \begin{bmatrix} I & -X\\-X^T & I\end{bmatrix} \psd 0,\;
        \begin{bmatrix} -X_{11} + X_{22} & -X_{21} - X_{12}\\-X_{21} - X_{12} & X_{11} - X_{22}\end{bmatrix} \nsd 0\bigg\}.\]
Since $\left[\begin{smallmatrix} -X_{11} + X_{22} & -X_{21} - X_{12} \\ -X_{21} - X_{12} & X_{11} - X_{22}\end{smallmatrix}\right]$ has 
trace zero, if it is also negative semidefinite then it must actually be zero. Consequently if $X\in \conv\,SO(2)$ then it must 
satisfy $X_{11} = X_{22}$ and $X_{12} = -X_{21}$ and so 
be of the form $X = \left[\begin{smallmatrix} c & -s\\s & c\end{smallmatrix}\right]$ for some $c,s\in \R$. Hence 
\[ \conv\,SO(2) = \left\{\begin{bmatrix} c & -s \\s & c\end{bmatrix} \in \R^{2\times 2}:
        \begin{bmatrix} 1 & 0 & -c & s\\0 & 1 & -s & -c\\-c & -s & 1 & 0\\s & -c & 0 & 1\end{bmatrix} \psd 0\right\}.\]
This is still a spectrahedral representation of size $4$, but the constraint has symmetry---it is invariant under 
simultaneously reversing the order of the rows and columns---suggesting that it can be block diagonalized \cite{gatermann2004symmetry}.
Under the change of coordinates
\begin{equation}
\label{eq:kk}
\frac{1}{2}\begin{bmatrix} 1 & 0 & -1 & 0\\0 & 1 & 0 & 1\\0 & 1 & 0 & -1\\-1 & 0 & -1 & 0\end{bmatrix}\!\!\!
    \begin{bmatrix} 1 & 0 & -c & s\\0 & 1 & -s & -c\\-c & -s & 1 & 0\\s & -c & 0 & 1\end{bmatrix}\!\!\!
    \begin{bmatrix} 1 & 0 & -1 & 0\\0 & 1 & 0 & 1\\0 & 1 & 0 & -1\\-1 & 0 & -1 & 0\end{bmatrix}^T\!\!\!\! = 
 \begin{bmatrix} 1 + c & s & 0 & 0\\s & 1-c & 0 & 0\\0 & 0 & 1+c & s\\0 & 0 & s & 1-c\end{bmatrix}\end{equation}
we see that the size $4$ spectrahedral representation in~\eqref{eq:kk} is actually two copies of the same 
size $2$ representation, allowing us to conclude that
\[ \conv\,SO(2) = \bigg\{\begin{bmatrix} c & -s\\s & c\end{bmatrix}\in \R^{2\times 2}: 
        \begin{bmatrix} 1+c & s\\s &1-c\end{bmatrix} \psd 0\bigg\}\]
as stated in Theorem~\ref{thm:son}.
\end{proof}

    \section{Lower bounds on the size of representations}
    \label{sec:lower}
    \subsection{Spectrahedral representations}
    Whenever a convex set has a polyhedral section, we can immediately obtain a
    simple lower bound on the possible size of a spectrahedral representation
    of that convex set in terms of the number of facets of that polyhedron.
\begin{lemma}
    Suppose $C\subset \R^n$ has a spectrahedral representation of size $m$ and
    $V$ is a subspace of $\R^n$ such that $C\cap V$ is a polytope with $f$
    (irredundant) facets. Then $m \geq f$.
\end{lemma}

\begin{proof}
    Suppose $C$ has a spectrahedral representation $C = \{x: \sum_i A_i x_i +
    A_0 \psd 0\}$ of size $m$, so the matrices $A_i$ are $m\times m$.  Then
    $p(x) = \det\left(\sum_i A_i x_i + A_0\right)$ is a polynomial of degree at
    most $m$ that vanishes on the boundary of $C$.  If $V$ is any subspace of
    $\R^n$ then $\left.p\right|_{V}$ is a polynomial of degree at most the
    degree of $p$ that vanishes on the boundary of $C\cap V$. Finally, any
    polynomial that vanishes on the boundary of a polyhedron with $f$
    (irredundant) facets has degree at least $f$ (since it must have a linear
    factor for each facet-defining hyperplane).  Consequently we have the chain
    of inequalities
\[ f\leq \textup{deg}(p|_{V}) \leq \textup{deg}(p) \leq m\]
establishing the result.
\end{proof}

Remarkably this simple technique allows us to establish that our spectrahedral
representations are of minimum size.
\begin{proof}[Proof of Theorem~\ref{thm:lower}]
    The diagonal slice of $O(n)^\circ$ is the cross-polytope, which (for $n\geq
    1$) has $2^n$ facets. Hence, for $n\geq 1$, any spectrahedral
    representation of $O(n)^\circ$ has size at least $2^n$.  The diagonal slice
    of $SO(n)^\circ$ is the polar of the parity polytope, which (for $n\geq 2$)
    has $2^{n-1}$ facets.  Hence, for $n\geq 2$, any spectrahedral
    representation of $SO(n)^\circ$ has size at least $2^{n-1}$.  The diagonal
    slice of $\conv\,SO(n)$ is the parity polytope, which for $n\geq 4$ has
    $2^{n-1}+2n$ facets, and for $n=3$ has $4$ facets.  It follows that any
    spectrahedral representation of $\conv \,SO(n)$ has size at least
    $2^{n-1}+2n$ for $n\geq 4$ and size at least $4$ for $n=3$. 
\end{proof}

The spectrahedral representations we construct in Section~\ref{sec:reps}
achieve these lower bounds and so are of minimum size. 

\subsection{Equivariant \psdlifts{}} 
\label{sec:eqpsd}
As is established in Theorem~\ref{thm:lower}, our spectrahedral representations
are necessarily of exponential size. While they are useful in practice for very
small $n$ (such as the physically relevant $n=3$ case), this is not the case
for larger $n$.

\paragraph{\psdlifts{}}
In general if $C$ is a spectrahedron, it may be possible to give a much smaller
\emph{projected spectrahedral} representation of $C$. In other words, it may be
the case that $C = \pi(D)$ where $\pi$ is a linear map\footnote{In this section
only, to conform with standard notation for \psdlifts{}, we use $\pi$ to mean
an arbitrary linear map} and $D$ has a spectrahedral representation that has
much smaller size then any spectrahedral representation of $C$. Note that
throughout this section if $D$ has a spectrahedral representation of size $m$
we express it as $D = \pi(L\cap \Sym_+^m)$ where $L$ is an affine subspace of
$\Sym^m$, the space of $m\times m$ symmetric matrices, and $\Sym_+^m\subset
\Sym^m$ is the cone of positive semidefinite $m\times m$ symmetric matrices.
The following definition is a specialization of~\cite[Definition
2.1]{gouveia2013lifts}.
\begin{definition}
    Suppose $C\subset \R^n$ is a convex body. If $C = \pi(L\cap \Sym_+^m)$
    where $L$ is an affine subspace of $m\times m$ symmetric matrices and
    $\pi:\Sym^m\rightarrow \R^n$ is a linear map, we say that $C$ has a
    \emph{\psdlift{}} of size $m$.
\end{definition}
It is straightforward to show that if $C$ has a \psdlift{} of size $m$, then
$C^\circ$ also has a \psdlift{} of size $m$ \cite{gouveia2013lifts}.  This
simple observation already yields examples of convex bodies for which there is
an exponential gap between the size of the smallest spectrahedral
representation and the size of the smallest \psdlift{}. For instance, as
demonstrated in Example~\ref{eg:onpolarproj}, the smallest possible
spectrahedral representation of $O(n)^\circ$ has size $2^n$ and yet it has a
\psdlift{} of size $2n$. 

\paragraph{Equivariant \psdlifts{}}
    While there has been some recent progress in obtaining lower bounds on the
    size of \psdlifts{} of some polytopes
    \cite{gouveia2013worst,briet2013existence}, little is understood about
    lower bounds on the size of \psdlifts{} of convex bodies in general.
    Recently, new techniques have been developed for obtaining lower bounds on
    the size of \emph{equivariant} \psdlifts{} of orbitopes. These are
    \psdlifts{} that `respect' (in a precise sense to be defined below) the
    symmetries of that orbitope. 
    
    In the remainder of this section we show that any projected spectrahedral
    representation of $\conv\,SO(n)$ that is equivariant with respect to the
    action of $S(O(n)\times O(n))$, must have size exponential in $n$. The
    argument works by showing that from any \psdlift{} of $\conv\,SO(n)$ that
    is equivariant with respect to the action of $S(O(n)\times O(n))$ we can
    construct a \psdlift{} of the parity polytope that is equivariant with
    respect to a certain group action on $\R^n$.  We then apply a recent result
    that gives an exponential lower bound on the size of appropriately
    equivariant \psdlifts{} of the parity polytope.

    The following definition \cite[Definition 2]{fawzi2013equivariant} makes
    the notion of equivariant \psdlift{} precise.
    \begin{definition}
        \label{def:eqpsdlift}
        Let $C \subset \R^n$ be a convex body invariant under the action of a
        group $G$ by linear transformations. Assume $C = \pi(L\cap \Sym^m_+)$
        is a \psdlift{} of $C$ of size $m$. The lift is called
        \emph{$G$-equivariant} if there is a group homomorphism $\rho : G
        \rightarrow GL(m)$ such that
        \[ \rho(g) X \rho(g)^T \in L \quad \forall g \in G, \; \text{for all $X \in L$} \]
and
\begin{equation}
 \label{eq:linequivariance}
 \pi(\rho(g) X \rho(g)^T) = g\cdot \pi(X) \quad \text{for all $g \in G$ and all $X \in L\cap \Sym^d_+$}.
\end{equation}
    \end{definition}
    In the present setting we are interested in two particular cases of
    equivariant \psdlifts{}: $S(O(n)\times O(n))$-equivariant \psdlifts{} of
    $\conv\,SO(n)$, and $\Gamma_{\parity}$-equivariant \psdlifts{} of the
    parity polytope. Here $\Gamma_\parity$ can be thought of concretely as the
    group of evenly signed permutation matrices---signed permutation matrices
    where there are an even number of entries that take the value $-1$. These
    act on $\R^n$ by matrix multiplication.

    We are now in a position to relate $S(O(n)\times O(n))$-equivariant
    \psdlifts{} of $\conv\,SO(n)$ with $\Gamma_\parity$-equivariant
    \psdlifts{} of $\PP_n$. 
    
    \begin{proposition}
        \label{prop:sonpplower}
        If $\conv\,SO(n)$ has an equivariant \psdlift{} of size $m$ then
        $\PP_n$ has an equivariant \psdlift{} of size $m$.     
    \end{proposition}
    \begin{proof}
        Suppose $\conv\,SO(n) = \pi(L\cap \Sym_+^m)$ is a $S(O(n)\times
        O(n))$-equivariant \psdlift{} of $\conv\,SO(n)$ of size $m$ and let
        $\rho:S(O(n)\times O(n))\rightarrow GL(m)$ be the associated
        homomorphism. Since the projection of $\conv\,SO(n)$ onto the subspace
        of diagonal matrices is $\PP_n$ (Theorem~\ref{thm:horn}) it follows
        that \[ \PP_n = (\projto{\Diag}\circ \pi)(L\cap \Sym_+^m)\] is a
        \psdlift{} of $\PP_n$ of size $m$. It remains to show that this lift of
        $\PP_n$ is $\Gamma_\parity$-equivariant. In other words we need to
        construct a homomorphism $\tilde{\rho}:\Gamma_\parity\rightarrow GL(m)$
        satisfying the requirements of Definition~\ref{def:eqpsdlift}.

        First observe that any element of $\Gamma_\parity$ can be uniquely
        expressed as as $DP$ where $D$ is a diagonal sign matrix with
        determinant one, and $P$ is a permutation matrix. Furthermore, note
        that if $D_1P_1$ and $D_2P_2$ are elements of $\Gamma_\parity$, then 
    \[ (D_1P_1)(D_2P_2) = (D_1P_1D_2P_1^T)(P_1P_2)\]
    gives the associated factorization of the product. Hence define
    $\phi:\Gamma_\parity\rightarrow S(O(n)\times O(n))$ by $\phi(DP) = (DP,P)$.
    Observe that this is a homomorphism because 
    \[\phi((D_1P_1)(D_2P_2)) = \phi((D_1P_1D_2P_1^T)(P_1P_2)) = ((D_1P_1)(D_2P_2),P_1P_2) =  \phi(D_1P_1)\cdot\phi(D_2P_2).\]
    Define a homomorphism $\tilde{\rho}: \Gamma_\parity\rightarrow GL(m)$ by
    $\tilde{\rho} = \rho \circ \phi$. For any symmetric matrix $X$ it is the
    case that $DP\cdot \projto{\Diag}(X) = \projto{\Diag}(DPXP^T)$. Hence  
    \begin{align*}
        DP\cdot\projto{\Diag}(\pi(X)) & = \projto{\Diag}(DP\pi(X)P^T)\\
                                   & = \projto{\Diag}(\phi(DP,P)\cdot \pi(X))\\
                                   & = \projto{\Diag}(\pi(\rho(DP,P)X\rho(DP,P)^T))\quad\text{since the lift of $\conv\,SO(n)$ is equivariant}\\
                                   & = \projto{\Diag}(\pi(\tilde{\rho}(DP) X \tilde{\rho}(DP)^T))\quad\text{by the definition of $\tilde{\rho}$}
    \end{align*}
    establishing that the lift is $\Gamma_\parity$-equivariant. 
    \end{proof}
    The following lower bound on the size of $\Gamma_\parity$-equivariant
    \psdlifts{} of the parity polytope is one of the main results
    of~\cite{fawzi2013equivariant}.
    \begin{theorem}
        \label{thm:pplower}
        Any $\Gamma_{\parity}$-equivariant \psdlift{} of $\PP_n$ for $n\geq 8$
        must have size at least $\frac{1}{n+1} 2^{0.54 n}$.
    \end{theorem}
    Combining Proposition~\ref{prop:sonpplower} with
    Proposition~\ref{thm:pplower} we obtain the following exponential lower
    bound on the size of any equivariant \psdlift{} of $\conv\,SO(n)$.
    \begin{corollary}
        Any $S(O(n)\times O(n))$-equivariant \psdlift{} of $\conv\,SO(n)$ for
        $n\geq 8$ must have size at least $\frac{1}{n+1} 2^{0.54 n}$.
    \end{corollary}

\section{Summary and open questions}
\label{sec:conclusion}
In this work we have constructed minimum-size spectrahedral representations for
the convex hull of $SO(n)$ and its polar.  We have also constructed a
minimum-size spectrahedral representation of $O(n)^\circ$ (the nuclear norm
ball).  We conclude the paper by discussing  some natural questions raised by
our results.

\subsection{Doubly spectrahedral convex sets}
We have seen that both the convex hull of $SO(n)$ and its polar are
spectrahedra. The same is true of the convex hull of $O(n)$ (the operator norm
ball) and its polar (the nuclear norm ball), as established by Sanyal et
al.~\cite[Corollary 4.9]{sanyal2011orbitopes}.  This is a very special
phenomenon---the polar of a spectrahedron is not, in general, a spectrahedron.
For example, the intersection of the second-order cone $\{(x,y,z): z \geq
\sqrt{x^2+y^2}\}$ and the non-negative orthant is a spectrahedron, but its
polar has non-exposed faces and so is not a
spectrahedron~\cite{ramana1995some}.

If a convex set $C$ and its polar are both spectrahedra, we say that $C$ is a
\emph{doubly spectrahedral} convex set.  Apart from $\conv\,O(n)$ and
$\conv\,SO(n)$, two distinct families of doubly spectrahedral convex sets are
the following:
\begin{description}
    \item[Polyhedra] Every polyhedron is a spectrahedron, and the polar of 
        a polyhedron is again a polyhedron. Hence polyhedra are doubly spectrahedral.
    \item[Homogeneous cones] A convex cone $K$ is \emph{homogeneous} if the automorphism group
        of $K$ acts transitively on the interior of $K$.  Using Vinberg's
        classification of homogeneous cones in terms of $T$-algebras
        \cite{vinberg1963theory}, Chua gave spectrahedral representations for
        all homogeneous cones \cite{chua2003relating}. Furthermore, $K$ is
        homogeneous if and only its dual cone $K^* = -K^\circ$ is homogeneous
        \cite[Proposition 9]{vinberg1963theory}. From these two observations it
        follows that any homogeneous cone is doubly spectrahedral.
\end{description}

We have seen that the doubly spectrahedral convex sets are a strict subset of
all spectrahedra that includes all polyhedra, all homogeneous convex cones, and
$\conv\,O(n)$ and $\conv\,SO(n)$. 
\paragraph{Problem} Characterize doubly spectrahedral convex sets.

\subsection{Non-equivariant \psdlifts{}}

In Section~\ref{sec:lower} we showed that our spectrahedral representations of $\conv\,SO(n)$ and $SO(n)^\circ$
are necessarily of exponential size
and that any $S(O(n)\times O(n))$-equivariant \psdlift{} of $\conv\,SO(n)$ must also have exponential size. 
    Our lower bound on the size of $S(O(n)\times O(n))$-equivariant \psdlifts{} of $\conv\,SO(n)$ used the 
    fact that any $\Gamma_\parity$-lift of the parity polytope has exponential size. Nevertheless, the parity
    polytope is known to have a \psdlift{} (in fact it is an LP lift) of size $4(n-1)$ \cite[Section 2.6.3]{carr2005polyhedral} that is \emph{not}
    $\Gamma_\parity$-equivariant (see \cite[Appendix C]{fawzi2013equivariant} for further discussion). It is quite possible that  
    by appropriately breaking symmetry we can find a small \psdlift{} of $\conv\,SO(n)$. 

\paragraph{Question} Does $\conv\,SO(n)$ have a \psdlift{} with size polynomial in $n$?

\appendix
\section{Clifford algebras and $\Spin(n)$}
\label{app:clif}
In this section we describe and establish the key properties of the quadratic
mapping $Q$ from Proposition~\ref{prop:spin} that underlies our spectrahedral
representation of $SO(n)^\circ$ given in Theorem~\ref{thm:sonpolar}. The
mapping $Q$ is most naturally described in terms of an algebraic structure
known as a Clifford algebra, which generalizes some properties of  complex
numbers and quaternions.  The first part of this section is devoted to
describing the basic properties of Clifford algebras we require. In
Section~\ref{sec:quadmap} we describe the mapping $Q$ and some of its
properties.  In Section~\ref{sec:spinprop} we define the set $\Spin(n)$ and
establish enough of its properties to prove
Propositions~\ref{prop:spin},~\ref{prop:basis},
and~\ref{prop:spinequivariance}. We prove Proposition~\ref{prop:mef} in
Section~\ref{sec:mef}.

Many of the constructions and properties we describe here are standard and can
be found, for example, in~\cite{atiyah1964clifford,harvey1990spinors}. We
highlight those aspects of the development that are novel as they arise.

\subsection{Clifford algebras}
\label{sec:clifford}
\paragraph{Definition}
The Clifford algebra $\Cl(n)$ is the associative algebra\footnote{That such an
algebra exists and is unique up to isomorphism follows because it can be
realized as a quotient of the tensor algebra.} (over the reals) with generators
$e_1,e_2,\ldots,e_n$ and relations 
\begin{equation}
    \label{eq:clifrel}
    e_i^2 = -\cliffid\quad\text{and}\quad e_ie_j = -e_je_i.
\end{equation}
Here $\cliffid$ denotes the multiplicative identity in the algebra. 

\paragraph{Standard basis}
As a real vector space $\Cl(n)$ has dimension $2^{n}$. A basis for $\Cl(n)$ is
given by all elements of the form \[ e_I := e_{i_1}e_{i_2}\cdots e_{i_k}\]
where $I = \{i_1,i_2,\ldots,i_k\}$ is a subset of $[n]$ and $i_1<i_2<\cdots <
i_k$. Here $e_{\emptyset}:=\cliffid$ is the multiplicative identity element in
$\Cl(n)$. Let us call $(e_I)_{I\subset[n]}$ the \emph{standard basis} for
$\Cl(n)$. With respect to this basis we can think of an arbitrary element $x\in
\Cl(n)$ as \[ x = \sum_{I\subset [n]} x_I e_I\] where the $x_I\in \R$. We equip
$\Cl(n)$ with the inner product $\langle x,y\rangle = \sum_{I\subset [n]}
x_Iy_I$. Clearly the standard basis is orthonormal with respect to this inner
product.

\paragraph{Left and right multiplication}
Any element $x\in \Cl(n)$ acts linearly on $\Cl(n)$ by left multiplication and
by right multiplication.  In other words, given $x\in \Cl(n)$ there are linear
maps $\lambda_x,\rho_x: \Cl(n)\rightarrow \Cl(n)$ defined by $\lambda_x(y) =
xy$ and $\rho_x(y) = yx$ for all $y\in \Cl(n)$. 

It is clear from the relations~\eqref{eq:clifrel} that $\lambda_{e_i}$ and
$\rho_{e_j}$ act on the standard basis of $\Cl(n)$ by signed permutations.
Specifically, if $I,J\subset [n]$, the corresponding entry of the signed
permutation matrix $[\lambda_{e_i}]$ is 
\[ [\lambda_{e_i}]_{I,J} = \begin{cases}
    (-1)^{|\{k\in I: k \leq i\}|} & \text{if $J = I\symdiff\{i\}$}\\
0 & \text{otherwise}\end{cases}\]
        and the $I,J$ entry of $[\rho_{e_i}]$ is
\[ [\rho_{e_i}]_{I,J} = \begin{cases}
    (-1)^{|\{k\in I: k \geq i\}|} & \text{if $J = I\symdiff\{i\}$}\\
0 & \text{otherwise}\end{cases}\]
        (where for two sets $I,J\subset [n]$, $I\symdiff J = (I\setminus
        J)\cup(J \setminus I)$ is their symmetric difference).  Hence both
        $\lambda_{e_i}$ and $\rho_{e_i}$ are skew-symmetric. Concise
        descriptions of these matrices that are particularly useful for
        implementation are
\begin{align*}
[\lambda_{e_i}] & = \overbrace{\Rtwo \otimes \cdots \otimes \Rtwo}^{i-1} \otimes \begin{bmatrix} 0 & -1\\1 & 0\end{bmatrix}\otimes 
    \overbrace{\Itwo \otimes \cdots \otimes \Itwo}^{n-i}\quad\text{and}\\
[\rho_{e_i}] & = \underbrace{\Itwo \otimes \cdots \otimes \Itwo}_{i-1} \otimes \begin{bmatrix} 0 & -1\\1 & 0\end{bmatrix}\otimes 
    \underbrace{\Rtwo \otimes \cdots \otimes \Rtwo}_{n-i}.
\end{align*}

\paragraph{Conjugation}
Observe that since $\lambda_{e_i}^* = -\lambda_{e_i} = \lambda_{-e_i}$ the
adjoint of left multiplication by $e_i$ is left multiplication by $-e_i$.
Similarly the adjoint of right multiplication by $e_i$ is right multiplication
by $-e_i$. In fact, it is the case that for any $x\in \Cl(n)$ there is
$\conj{x}\in \Cl(n)$ such that $\lambda_x^* = \lambda_{\conj{x}}$ and
$\rho_{x}^* = \rho_{\conj{x}}$. To see this define a \emph{conjugation} map
$x\mapsto \conj{x}$ on the standard basis by
\[ \conj{e_I} = (-1)^{|I|}e_{i_k}\cdots e_{i_2}e_{i_1}\]
and extend by linearity. It is easy to see by direct computation that
$\lambda_{e_I}^* = \lambda_{\conj{e_I}}$ and $\rho_{e_I}^* = \rho_{\conj{e_I}}$
as required.  We use this conjugation map repeatedly in the sequel, usually via
the relations
\begin{equation}
    \label{eq:leftadj}\langle xy,z\rangle = \langle \lambda_xy,z\rangle = \langle y,\lambda_x^*z\rangle = \langle y,\lambda_{\conj{x}}z\rangle = \langle y,\conj{x}z\rangle
\end{equation}
and  
\begin{equation}
\label{eq:rightadj}
\langle yx,z\rangle = \langle \rho_xy,z\rangle = \langle y,\rho_x^*z\rangle = \langle y,\rho_{\conj{x}}z\rangle = \langle y,z\conj{x}\rangle.
\end{equation}

\paragraph{Copy of $\R^n$ in $\Cl(n)$}
Throughout this appendix, we use the notation $\R^n$ to denote the
$n$-dimensional subspace of $\Cl(n)$ spanned by the generators
$e_1,e_2,\ldots,e_n$, and the notation $S^{n-1}\subset \R^{n}$ to denote the
elements $x\in \R^n$ satisfying $\langle x,x\rangle = 1$. We next state and
prove some basic properties of the elements of $S^{n-1}\subset \Cl(n)$.
\begin{lemma}
    \label{lem:sphereinv}
    If $u\in S^{n-1}\subset \Cl(n)$ then $u \conj{u} = \cliffid$. Consequently
    $\langle uy,uz \rangle = \langle y,x\rangle = \langle yu,zu\rangle$ for all
    $y,z\in \Cl(n)$. 
\end{lemma}
\begin{proof}
    The second statement follows from the first together
    with~\eqref{eq:leftadj} and~\eqref{eq:rightadj}. That $u\conj{u} =
    \cliffid$ whenever $u\in S^{n-1}$ follows from a direct computation using
    the defining relations of $\Cl(n)$ from~\eqref{eq:clifrel}.
 \end{proof}
 The following can be established by repeatedly applying Lemma~\ref{lem:sphereinv}.
 \begin{corollary}
     \label{cor:sphere}
     If $u_1,u_2,\ldots,u_k\in S^{n-1}$ then $\langle u_1u_2\cdots u_k,u_1u_2\cdots u_k\rangle = 1$.
 \end{corollary}

\paragraph{Even subalgebra}
Consider the subspaces $\Cl^0(n)$ and $\Cl^1(n)$ of $\Cl(n)$ defined by
\[ \Cl^0(n) = \textup{span}\{e_I: \text{$I\subset [n]$, $|I|$ even}\}\quad\text{and}\quad
 \Cl^1(n) = \textup{span}\{e_I: \text{$I\subset [n]$, $|I|$ odd}\}.\] It is
 straightforward to show that if $x,y\in \Cl^0(n)$ then $xy\in \Cl^0(n)$, and
 if $x,y\in \Cl^1(n)$ then $xy\in \Cl^0(n)$.  The first of these properties
 states that $\Cl^0(n)$ is a subalgebra of $\Cl(n)$, which we call the
 \emph{even subalgebra}. With these properties we have that the product of an
 even number of elements of $S^{n-1}$ is in the even subalgebra.
\begin{lemma}
    \label{lem:even}
    If $u_1,u_2,\ldots,u_{2k}\in S^{n-1}$ then $x = u_1u_2\cdots u_{2k}\in \Cl^0(n)$.
\end{lemma}
\begin{proof}
    Since $S^{n-1}\subset \R^n\subset \Cl^1(n)$, each $u_i\in \Cl^1(n)$.  Hence
    $u_{2i-1}u_{2i}\in \Cl^0(n)$ for $i=1,2,\ldots,k$. So $u_1u_2\cdots u_{2k}
    = (u_1u_2)(u_3u_4)\cdots (u_{2k-1}u_{2k})$ is the product of elements in
    $\Cl^0(n)$ so is itself an element of $\Cl^0(n)$. 
\end{proof}

The final property of elements of $\R^n\subset \Cl(n)$ we use in the sequel is
the coordinate-free version of the defining relations of $\Cl(n)$ given
in~\eqref{eq:clifrel}. 
\begin{lemma}
    If $u,v\in \R^n$ then 
\begin{equation}
    \label{eq:gr1rel}
    uv+vu = -2\langle u,v\rangle\,\cliffid.
\end{equation}
\end{lemma}
\begin{proof}
    First note that~\eqref{eq:gr1rel} is bilinear in $u$ and $v$ so it suffices
    to verify the identity for $u=e_i$ and $v=e_j$ (for all $1\leq i,j\leq n$).
    That the statement holds for $u=e_i$ and $v=e_j$ (for all $1\leq i,j\leq
    n$) is equivalent to the relations~\eqref{eq:clifrel} (since 
    $\langle e_i,e_j\rangle = \delta_{ij}$). 
\end{proof}

\subsection{The quadratic mapping}
\label{sec:quadmap}
We now define and establish the relevant properties of the quadratic mapping $Q:\Cl^0(n)\rightarrow \R^{n\times n}$
that plays a prominent role in Section~\ref{sec:gen}. Our aim is to prove Proposition~\ref{prop:spin}.
First define $\tildeQ:\Cl(n)\rightarrow \R^{n\times n}$ by
\[ \tildeQ(x)(u) = \projto{\R^n}\lambda_x\rho_{\conj{x}}(u) = \projto{\R^n}(xu\conj{x}).\]
Note that $\tildeQ(x)$ is quadratic in $x$.
When we express the linear map $\tildeQ(x)$ as a matrix (with respect to the standard basis) we see that
    \[ [\tildeQ(x)]_{ij} = \langle e_i,xe_j\conj{x}\rangle.\]
Then define $Q:\Cl^0(n)\rightarrow \R^{n\times n}$ as the restriction of $\tildeQ$ to the subalgebra $\Cl^0(n)$. 

This construction is motivated by the fact that if $u\in S^{n-1}$ then $-\tildeQ(u)$ is the 
reflection in the hyperplane orthogonal to $u$. 
\begin{lemma}
    \label{lem:reflect}
    Let $u\in S^{n-1}$. Then whenever $v\in \R^n$, $-uv\conj{u}\in \R^n$ is the reflection of 
    $v$ in the hyperplane normal to $u$. In particular $-uv\conj{u}\in \R^n$. 
\end{lemma}
\begin{proof}
    Let $u\in S^{n-1}$. Then by~\eqref{eq:gr1rel}, if $v\in \R^n$ then $-uv =
    2\langle u,v\rangle \cliffid + vu$ and so since $u\conj{u}=\cliffid$ and
    $\conj{u}=-u$ it follows that
        \[
        -uv\conj{u}  = 2\langle u,v\rangle \conj{u} + vu\conj{u} = v-2\langle u,v\rangle u
    \]
    which is precisely the reflection in the hyperplane orthogonal to $u$ and
    is certainly in $\R^n$.
\end{proof}
Note that our definition of $\tildeQ$ is one possible extension to all of
$\Cl(n)$ of the map that sends $u\in S^{n-1}$ to the reflection in the
hyperplane orthogonal to $u$.  It is specifically chosen so as to be quadratic
on all of $\Cl(n)$. Our choice is different from the typical extension used in
the literature---the \emph{twisted adjoint representation}
\cite{atiyah1964clifford}--- which is \emph{not} quadratic in $x$ on all of
$\Cl(n)$ and is not suitable for our purposes.

\begin{lemma}
    \label{lem:hom}
    Let $x\in \Cl(n)$ and $u\in S^{n-1}$. Then 
    \[ \tildeQ(xu) = \tildeQ(x)\tildeQ(u)\quad\text{and}\quad
    \tildeQ(ux) = \tildeQ(u)\tildeQ(x)
    \]
    where the product on the right hand side of each equality is composition of linear maps.
\end{lemma}
\begin{proof}
    If $u\in S^{n-1}$, we know from the previous lemma that $v\mapsto
    uv\conj{u}$ leaves the subspace $\R^n$ (and hence its orthogonal
    complement) invariant. So by the definition of $\tildeQ$ we see that
    \[ \tildeQ(xu)(v) = \projto{\R^n}(xuv\conj{u}\,\conj{x}) = \projto{\R^n}(x\projtofrom{\R^n}(uv\conj{u})\conj{x}) = \tildeQ(x)(\tildeQ(u)(v)).\]
    Similarly since $\projtofrom{\R^n}+\projtofrom{{\R^n}^\perp} = I$, 
    \[ \tildeQ(ux)(v) = \projto{\R^n}(uxv\conj{x}\,\conj{u}) = \projto{\R^n}(u\projtofrom{\R^n}(xv\conj{x})\conj{u}) + 
    \projto{\R^n}(u\projtofrom{{\R^n}^\perp}(xv\conj{x})\conj{u}) = Q(u)(Q(x)(v)) + 0\]
    where we have used the fact that $uy\conj{u}\in {\R^{n}}^\perp$ whenever $y\in {\R^n}^\perp$.
\end{proof}

\subsection{$\Spin(n)$ and the proofs of Propositions~\ref{prop:spin},~\ref{prop:basis} and~\ref{prop:spinequivariance}}
\label{sec:spinprop}
\begin{definition}
    Define $\Spin(n)$ as the set of all
    even length products of elements of $S^{n-1}$:
    \[ \Spin(n) = \{x\in \Cl(n): x = u_1u_2\cdots u_{2k}\;\text{for some positive integer $k$ and $u_1,\ldots,u_{2k}\in S^{n-1}$}\}.
    \]
\end{definition}
Although we do not require this fact, it can be shown that in the above
definition it is enough to take $k= \lfloor n/2\rfloor$. We note that a common
alternative definition~\cite{atiyah1964clifford} is to take $\Spin(n)$ to be
the elements of $\Cl^0(n)$ satisfying $x\conj{x} = \cliffid$ and $xv\conj{x}
\in \R^n$ for every $v\in \R^n$ (which defines a real algebraic variety
specified by the vanishing of a collection of quadratic equations). It is
fairly straightforward to establish that these two definitions are equivalent.

The next result establishes that $\Spin(n)$ is a group under multiplication. 
\begin{lemma}
    \label{lem:group}
    If $x\in \Spin(n)$ then $\conj{x}x = x\conj{x} = \cliffid$. If $x,y\in
    \Spin(n)$ then $xy\in \Spin(n)$. 
\end{lemma}
\begin{proof}
    That $\Spin(n)$ is closed under multiplication is clear from the
    definition. That conjugation and inversion coincide on $\Spin(n)$ follows
    from Lemma~\ref{lem:sphereinv}.
\end{proof}
The following result establishes Proposition~\ref{prop:spin} and
Proposition~\ref{prop:basis}.
\begin{lemma}
    \label{lem:qonto}
    $\Spin(n)$ is a subset of the unit sphere in $\Cl^0(n)$, i.e.
    $\Spin(n)\subset \{x\in \Cl^0(n): \langle x,x\rangle = 1\}$, satisfying
    $Q(\Spin(n)) = SO(n)$. Furthermore whenever $I\subset [n]$ has even
    cardinality, $e_I \in \Spin(n)$. 
\end{lemma}
\begin{proof}
    That $\Spin(n)\subset \{x\in \Cl^0(n): \langle x,x\rangle =1\}$ follows
    directly from Lemma~\ref{lem:even} and Corollary~\ref{cor:sphere}.  Let
    $X\in SO(n)$. By the Cartan-Dieudonn\'{e} theorem \cite{gallier2001cartan}
    any such $X$ can be expressed as the composition of an even number (at most
    $n$) of reflections in hyperplanes with normal vectors, say,
    $u_1,u_2,\ldots,u_{2k}\in S^{n-1}$. Let $x = u_1u_2\cdots u_{2k-1}u_{2k}\in
    \Spin(n)$.  Then by Lemma~\ref{lem:reflect} and Lemma~\ref{lem:hom} and the
    fact that $Q$ is the restriction of $\tildeQ$ to $\Cl^0(n)$, 
    \[ X = \tildeQ(u_{1})\tildeQ(u_{2})\cdots \tildeQ(u_{2k-1})\tildeQ(u_{2k}) = \tildeQ(x) = Q(x)\in Q(\Spin(n)).\]
    Hence $SO(n)\subseteq Q(\Spin(n))$. On the other hand, if $x=u_1u_2\cdots
    u_{2k-1}u_{2k}\in \Spin(n)$ then $Q(x)$ is the product of an even number of
    reflections in hyperplanes and so is an element of $SO(n)$, establishing
    the reverse inclusion.
    
    For the last statement, let $I = \{i_1,\ldots,i_{2k}\}$ be a subset of
    $[n]$ with even cardinality and suppose $i_1<i_2<\cdots < i_{2k}$.  Then
    $e_I = e_{i_1}e_{i_2}\cdots e_{i_{2k}}$ realizes $e_I$ as the product of an
    even number of elements of $S^{n-1}$, showing that $e_I\in \Spin(n)$.
\end{proof}
We conclude the section by establishing
Proposition~\ref{prop:spinequivariance}.
\begin{lemma}
    If $U,V\in SO(n)$ then there is a corresponding invertible linear map
    $\Phi_{(U,V)}:\Cl^0(n)\rightarrow \Cl^0(n)$ such that for any $x\in
    \Cl^0(n)$, $UQ(x)V^T = Q(\Phi_{(U,V)}x)$ and $\Phi_{(U,V)}(\Spin(n)) =
    \Spin(n)$. 
\end{lemma}
\begin{proof}
    By Lemma~\ref{lem:qonto} there are $u,v\in \Spin(n)$ such that $Q(u) = U$
    and $Q(v) = V$. Define $\Phi_{(U,V)}:\Cl^0(n)\rightarrow \Cl^0(n)$ by
    $\Phi_{(U,V)}(x) = ux\conj{v}$. Then $\Phi_{(U,V)}$ is invertible with
    inverse $\Phi_{(U,V)}^{-1}(x) = \conj{u}xv$.  Since $Q(1) = I$, by
    Lemma~\ref{lem:hom} we have that whenever $v\in \Spin(n)$, $Q(v)$ is
    orthogonal and so
    \[ I = Q(1) = Q(v \conj{v}) = Q(v)Q(\conj{v}) = Q(v)Q(v)^T.\]
    Again by Lemma~\ref{lem:hom}, for any $x\in \Cl^0(n)$, 
    \[ UQ(x)V^T = Q(u)Q(x)Q(v)^T = Q(u)Q(x)Q(\conj{v}) = Q(ux\conj{v}).\]
    Finally, if $x\in \Spin(n)$ then $\Phi_{(U,V)}(x) = ux\conj{v} \in
    \Spin(n)$ by Lemma~\ref{lem:group}.  Hence $\Phi_{(U,V)}(\Spin(n)) =
    \Spin(n)$.
\end{proof}

\subsection{Matrices of the quadratic mapping (proof of Proposition~\ref{prop:mef})}
\label{sec:mef}
For $1\leq i,j\leq n$, let $A_{ij}$ be the $2^{n-1}\times 2^{n-1}$ symmetric
matrix representing the quadratic form $[Q(x)]_{ij}$, i.e.\
\[ [Q(x)]_{ij} = \langle e_i,xe_j\conj{x}\rangle = \sum_{I,J\in \Ieven}
x_Ix_J[A_{ij}]_{I,J}\]
where $\Ieven$ is the set of subsets of $[n]$ of even cardinality.  We now turn
to describing these matrices concretely, giving a proof of
Proposition~\ref{prop:mef}.

\begin{proof}[Proof of Proposition~\ref{prop:mef}]
    We first define matrices $\widetilde{A}_{ij}$ for $1\leq i,j \leq n$. These
    are the $2^n\times 2^n$ symmetric matrices representing the quadratic forms
    $[\tildeQ(x)]_{ij}$, i.e.\
\[ [\tildeQ(x)]_{ij} = \langle e_i,xe_j\conj{x}\rangle = \sum_{I,J\subseteq [n]}x_Ix_J[\widetilde{A}_{ij}]_{I,J}.\]
Since  
\[ \langle e_i,xe_j\conj{x}\rangle =\langle e_ix,xe_j\rangle = \langle x,\lambda_{\conj{e_i}}\rho_{e_j}x\rangle = 
    -\langle x,\lambda_{e_i}\rho_{e_j}x\rangle\]
it follows that 
\begin{equation}
    \label{eq:tildeA}
    \widetilde{A}_{ij} = -\lambda_{e_i}\rho_{e_j}.
\end{equation}
Recall that in Section~\ref{sec:clifford} we give concrete expressions for the
matrices defining $\lambda_{e_i}$ and $\rho_{e_j}$. Hence~\eqref{eq:tildeA}
gives a convenient way to explicitly build $\widetilde{A}_{ij}$. Note that
$\lambda_{e_i}\rho_{e_j}$ is symmetric because $\lambda_{e_i}$ and $\rho_{e_j}$
are skew-symmetric and $\lambda_{e_i}$ and $\rho_{e_j}$ commute (because the
operations of left- and right-multiplication commute in an associative
algebra).

We now obtain concrete expressions for the $A_{ij}$ (rather than the
$\widetilde{A}_{ij}$). Note that since $Q$ is the restriction of $\tildeQ$ to
the subspace $\Cl^0(n)$, so for each $1\leq i,j \leq n$, $A_{ij}$ is the
$2^{n-1}\times 2^{n-1}$ principal submatrix of $\widetilde{A}_{ij}$ indexed by
rows and columns corresponding to the basis elements $(e_I)_{\Ieven}$
of $\Cl^0(n)$. This submatrix can be extracted by computing
\[ A_{ij} =  P_{\textup{even}}^T\widetilde{A}_{ij}P_{\textup{even}}\]
where $P_{\textup{even}}$ is the $2^{n}\times 2^{n-1}$ matrix with exactly one
non-zero entry per column and at most one non-zero entry per row given by 
\begin{align*}
    P_{\textup{even}} & = \frac{1}{2}\begin{bmatrix} I_{2^{n-1}} + \Rtwo\otimes \cdots \otimes \Rtwo\\
    \phantom{\ldots}\\
        I_{2^{n-1}} - \Rtwo \otimes \cdots \otimes\Rtwo\end{bmatrix}
\end{align*}
which verifies~\eqref{eq:AB}.
\end{proof}

\bibliographystyle{plain}
\bibliography{son_bib}

\end{document}